\documentclass[10pt]{amsart}
\usepackage{graphicx}
\usepackage{amsmath}
\usepackage{amssymb}
\usepackage{enumerate}
\usepackage{amsthm}

\makeatletter
 \theoremstyle{plain}
\newtheorem{thm}{Theorem}[section]
  \theoremstyle{plain}
  
 \theoremstyle{definition}
  \newtheorem{example}[thm]{Example}
  \theoremstyle{definition}
  \newtheorem{defn}[thm]{Definition}
  \theoremstyle{definition}
  \newtheorem{rk}[thm]{Remark}
  \theoremstyle{definition}
  \theoremstyle{plain}
  
  \theoremstyle{plain}
  \newtheorem{lem}[thm]{Lemma}
    \theoremstyle{plain}
  
 \makeatother

\renewcommand{\epsilon}{\varepsilon}

\newcommand{\bE}{{\Bbb E}}

\newcommand{\F}{\mathcal{F}}

\renewcommand{\L}{\mathcal{L}}

\renewcommand{\O}{\mathcal{O}}
\renewcommand{\P}{\mathcal{P}}
\newcommand{\R}{\mathbb{R}}
\renewcommand{\S}{\mathbb{S}}

\newcommand{\Wi}{\mathbb{W}}

\renewcommand{\P}{{\Bbb P}}
\newcommand{\floor}[1]{{\lfloor #1 \rfloor}}

\makeatother

\begin{document}

\title[An Approx Scheme for Reflected Stochastic Differential Equations]{An Approximation Scheme for Reflected Stochastic Differential Equations}

\author{Lawrence Christopher Evans}
%    Address of record for the research reported here
\address{Department of Mathematics, Massachusetts Institute of Technology, Cambridge, Massachusetts 02139}
%    Current address
\curraddr{Department of Mathematics,
University of Missouri, Columbia, Missouri 65211}
\email{lcevans@math.mit.edu}

\author{Daniel W. Stroock}
%    Address of record for the research reported here
\address{Department of Mathematics, Massachusetts Institute of Technology, Cambridge, Massachusetts 02139}
%    Current address
%\curraddr{Department of Mathematics,
%University of Missouri, Columbia, Missouri 65211}
\email{dws@math.mit.edu}



\date{\today{}}

\begin{abstract} In this paper we consider the Stratonovich reflected
  stochastic differential equation
  $dX_t=\sigma(X_t)\circ dW_t+b(X_t)dt+dL_t$ in a bounded domain $\O$ which
  satisfies conditions, introduced by Lions and
  Sznitman, which are specified below. Letting $W^N_t$ be the $N$-dyadic piecewise linear
  interpolation of $W_t$ what we show is that one can solve the reflected
  ordinary differential equation
  $\dot X^N_t=\sigma(X^N_t)\dot W^N_t+b(X^N_t)+\dot L^N_t$ and that the distribution
  of the pair $(X^N_t,L^N_t)$ converges weakly to that of $(X_t,L_t)$.  Hence, 
  what we prove is a distributional version for reflected diffusions of the
  famous result of Wong and Zakai.

Perhaps the most valuable contribution made by our procedure derives from
the representation of $\dot X^N_t$ in terms of a projection of $\dot
W_t^N$.  In particular, we apply our result in hand to derive some geometric properties
of coupled reflected Brownian motion in certain domains, especially
those properties which have been used in recent work on the ``hot spots''
conjecture for special domain.\end{abstract}

\maketitle

\thispagestyle{empty}

\setcounter{page}{1}

%%%%%FirstSection%%%%%%

\section{Introduction}

\subsection{Motivation}

As is well known, It\^o stochastic differential equations can be very
misleading from a geometric standpoint.  The classic example of this
observation is the It\^o stochastic differential equation (SDE)
$$dX(t)=\sigma \bigl(X(t)\bigr)dW_t\text{ with } X(0)=\begin{pmatrix}
1\\0\end{pmatrix}\text{ and } \sigma =\begin{pmatrix} -x_2\\x_1\end{pmatrix},$$
where $W_t$ is a $1$-dimensional Brownian motion.  If one makes the mistake
of thinking that It\^o differentials of Brownian motion behave like classical differentials,
then one would predict the $X(t)$ should live on the unit circle.  On the
other hand, It\^o's formula, which is a quantitative statement of the
extent to which they do not behave like classical differentials, says that
$d|X(t)|^2=|X(t)|^2dt$, and so $|X(t)|^2=e^t$.

To avoid the sort of misinterpretation to which It\^o SDE's lead, it is
convenient to replace It\^o SDE's by their Stratonvich counterparts.  When
one does so, then the Wong--Zakai theorem ~\cite{WZ} shows that the
solution to the SDE can be approximated by solutions to the ordinary
differential equation (ODE) which one obtains by piecewise linearizing the
Brownian paths.  In this way, one can transfer to solutions of the SDE
geometric properties which one knows for the solutions to the ODE's.  The
purpose of this paper is to carry out the analogous program for SDE's for 
diffusions which are reflected at the boundary of some region.  This is not
the first time that such a program has been attempted.  For example,
R. Petterson  proved in ~\cite{Peter} a result of this sort under the
assumption that the domain is convex.  Unfortunately, convexity is too
rigid a requirement for applications of the sort which appear in papers
like ~\cite{HotSpots} by Banuelos and Burdzy, and so it is important to
replace convexity by a more general condition, like the one given in
~\cite{LS} by A. Sznitman and P.L. Lions.  Finally, it should be mentioned
that the article ~\cite{Arturo} by A. Kohatsu-Higa contains a very general,
highly abstract approximation procedure which may be applicable to the
situation here.   

\subsection{Background for Reflected SDE's}

We begin by recalling the (deterministic) Skorohod problem.

Let $\O\subset \R^d$ be a domain and to each $x\in\partial\O$ assign a nonempty
collection $\nu(x)\subseteq\S^{d-1}$, to be thought of as the set of
directions in which a path can be ``pushed'' when it hits $x$.  Given a
continuous path
$w:[0,\infty)\rightarrow \R^d$ with $w_0\in\O$, known as the ``input,'' we say that a
\emph{solution to the Skorohod problem for} $(\O,\nu(x))$ is a pair
$(x_\cdot,\ell_\cdot)$ consisting of a continuous path
$t\in[0,\infty)\longmapsto x_t\in\bar{\O}$ and a continuous function of
locally bounded variation $t\in[0,\infty )\longmapsto \ell_t\in\R^d$ such that
      \begin{equation} x_t=w_t+\ell_t,\; |\ell|_t=\int_0^t
        1_{\partial\O}(x_s)d|\ell|_s,\text{ and
        }\ell_t=\int_0^t\nu(x_s)d|\ell|_s, \end{equation}
where $|\ell|_t$ denotes the total variation of $\ell_t$ on the interval $[0,t]$, and
the third line is a shorthand way of saying that
$$\frac{d\ell_t}{d|\ell|_t}\in\nu(x_t),\quad d|\ell|_t-\text{a.e.}.$$
When a unique solution exists for each input, we will call the map
$w_\cdot\rightsquigarrow (x_\cdot,\ell_\cdot)$ the \emph{Skorohod
map} and will denote it by $\Gamma $.  Also, the path $x_\cdot$ will be
referred to as the ``output.''

Throughout this paper we will take $\nu(x)$ to be the collection of inward pointing
proximal normal vectors
\begin{equation} \label{defnofnu} \nu(x)\equiv \big\{\nu\in
  \S^{d-1}:\exists C>0\;\forall x'\in\bar\O\;(x'-x)\cdot\nu+ C|x-x'|^2\geq 0\}.  \end{equation}
Elementary algebra shows that
\begin{equation} \label{elementaryalg} (x'-x)\cdot\nu+ C|x-x'|^2< 0
  \Longleftrightarrow \left|x'-(x-\frac{\nu}{2C})\right|^2<\left(\frac{1}{2C}\right)^2
\end{equation} which shows that, geometrically, $\nu(x)$ is the collection
of unit vectors based at $x\in\partial\O$ such that there exists an open
ball touching the base of $\nu$ but not intersecting $\O$.

The class of domains which we will consider was described by Lions and
Sznitman in~\cite{LS}.  Namely, we will say that $\O$ is \emph{admissible} if

\begin{defn}
\label{admissible}
\begin{enumerate}
\item $\forall x\in\partial\O$, $\nu(x)\neq\phi$, and there exists a $C_0\geq 0$
such that 
$$
(x'-x)\cdot\nu+C_0|x-x'|^2\geq 0\text{ for all }x'\in\O,\;x\in\partial
\O,\text{ and }\nu\in\nu(x).$$

\item There exists a function $\phi\in C^2(\R^d;\R)$ and $\alpha >0$ such that 
$$
\nabla\phi(x)\cdot\nu\geq \alpha>0\text{ for all }x\in\partial\O\text{ and }\nu\in\nu(x).$$

\item There exist $n\geq 1$, $\lambda > 0$, $R>0$,
  $a_1,\ldots,a_n\in\S^{d-1}$, and $x_1,\ldots,x_n\in\partial O$ such that  
$$\begin{gathered}
\partial O\subseteq\bigcup_{i=1}^nB(x_i,R)\text{ and }\\
x\in\partial O\cap B(x_i,2R)\Longrightarrow\nu\cdot a_i\geq \lambda >
0\text{ for all }\nu\in\nu(x).\end{gathered}$$\end{enumerate}\end{defn}

In view of (\ref{elementaryalg}), Part 1 of Definition \ref{admissible} can
be seen as a sort of uniform exterior ball condition. More precisely, it
says that not only can every point $x\in\partial\O$ can be touched by an
exterior ball but also that the exterior ball touching $x$ can be scaled to
have a uniformly large radius.
In the convex analysis literature, the closure of a set $\O$ satisfying Part 1 of
Definition \ref{admissible} is said to be \emph{uniformly prox-regular}
(See~\cite{PRT}, especially Theorem 4.1, for more on the properties of
uniformly prox-regular sets).

Parts 2 and 3 of Definition \ref{admissible} are regularity
requirements on $\partial\O$ which ensure that the ``normal vectors'' don't
fluctuate too wildly. In this connection, notice that Part 3 is implied by
Part 2 when $\O$ is bounded.

In their paper~\cite{LS}, Lions and Sznitman show that for each $w_\cdot\in
C\bigl([0,\infty );\R^d\bigr)$ there exists an almost surely 
unique solution $(x_\cdot,\ell_\cdot)$
to the deterministic Skorohod problem when the domain $\O$ is admissible.
The map $\Gamma $ which takes $w_\cdot$ to $x_\cdot$ is called the deterministic
Skorohod map.

We turn next to the formulation of reflected diffusions in terms of a
Skorohod problem for an SDE.  Until further notice, we will be looking at
It\^o SDE's and will only reformulate them as Stratonovich SDE's when it is
important to do so.

Let $\O\subset \R^d$ an admissible domain, and let $\sigma
:\bar\O\longrightarrow \rm{Hom}(\R^r;\R^d)$ and $b:\bar\O\longrightarrow
\R^d$ be uniformly Lipschitz continuous maps.  Given an $r$-dimensional
Brownian motion $W_\cdot$ and $x_0\in\O$, a solution to 
$(X_\cdot,L_\cdot)$ to the reflected SDE (\ref{reflectedSDE}) is a continuous
process $\bigl\{(X_t,L_t):\,t\ge0\bigr\}$ which is progressively measurable
with respect to $W_\cdot$ and satisfies the conditions that $(X_t,L_t)\in
\bar\O\times \R^d$ and $|L|_t<\infty $ for all $t\ge0$, and, almost surely,
    \begin{equation} \label{reflectedSDE} \begin{split}
        X_t=&x_0+\int_0^t\sigma(X_s)dW_s+\int_0^t
        b(X_s)ds+L_t,\\ |L|_t=&\int_0^t 1_{\partial\O}(X_s)d|L|_s,\text{
          and }L_t=\int_0^t\nu(X_s)d|L|_s, \end{split} \end{equation} where
    $|L|_t$ denotes the total variation of $L_t$ by time $t$, and the third
    line is shorthand for $\frac{dL_t}{d|L|_t}\in\nu(x_t),\;d|L|_t-\text{a.e.}$.  

Existence and uniqueness of solution to reflected SDE's was proved by
H. Tanaka in ~\cite{Tanaka} when $\O$ is convex.  The extension of his
result to admissible domains was made by Lions and Sznitman in ~\cite{LS}
and Saisho in ~\cite{Saisho}. We refer the reader to those papers for an
overview of the subject. 

%%%%%%%%%%%%%%%%%%%%%%%NEWSECTION%%%%%%%%%%%%%%%%%%%%%%%%%%%%%%%

\section{\label{2}  Equations with Reflection}

\subsection{Properties of Solutions to Reflected ODE's}

Suppose that $\O$ is a bounded, admissible domain and that $\sigma :\bar\O\longrightarrow
\rm{Hom}(\R^r;\R^d)$ is uniformly Lipschitz continuous.  In this section we
will show that, for each $x_0\in\bar{\O}$ and $w_\cdot\in C\bigl([0,\infty
  );\R^d\bigr)$ there is precisely one solution $(x_\cdot,\ell_\cdot)$ to
the reflected ODE
\begin{equation}
\label{reflectedODE}
\begin{split}
x_t=&x_0+\int_0^t\sigma(x_s)dw_s+\ell_t,\\
|\ell|_t=&\int_0^t 1_{\partial\O}(x_s)d|\ell|_s,\text{ and }\ell_t=\int_0^t\nu(x_s)d|\ell|_s,
\end{split}
\end{equation}
where $x_\cdot\in C\bigl([0,\infty );\bar \O\bigr)$ and $\ell_t:[0,\infty
    )\longrightarrow \R^d$ is a continuous function having finite variation
    $|\ell|_t$ on $[0,t]$ for all $t>0$.  In addition, we will give
    a geometrically appealing alternate description of this solution.
    Previously,  existence and uniqueness results for variants of
    (\ref{reflectedODE}) are well known in the convex analysis literature.
    For example, see~\cite{Bettiol} for a recent such result as well as a
    good overview of other known results. 

Although the proofs of existence and uniqueness are implicit in the
contents of other articles, we, mimicking the proof of Theorem 3.1 in
\cite{LS}, will prove them here.  For this purpose, consider the map
$F_w:C\bigl([0,\infty );\bar O\bigr)\longrightarrow C\bigl([0,\infty );\bar O\bigr)$
given by $F_w(y_\cdot)=\Gamma(x_0+\int_0^\cdot
{\sigma(y_s)dw_s})$, where $\Gamma$ is the Skorohod map.  We will show that
$F$ has a unique fixed point, and the key to doing so is contained in the next lemma.

\begin{lem}
\label{contraclem}  For each $T>0$ there exists a $C=C_w(T)<\infty $
such that for any pair of paths
$y_\cdot$ and $y'_\cdot$,
\begin{eqnarray}
\label{contractioninequality}
|F_w(y_\cdot)_t-F(y'_\cdot)_t|^2\leq
\int_0^te^{C(t-\tau )}|y_\tau -y'_\tau |^2d\tau \quad\text{for all }t\in[0,T]. 
\end{eqnarray}
\end{lem}

\begin{proof}  Set $z_\cdot=F(y_\cdot)$ and $z'_\cdot=F(y'_\cdot)$.  Given
  $T>0$, we will
  show that there is a $C<\infty $ such that
$$|z_t-z'_t|^2\leq C\left(\int_0^t|z_\tau -z'_\tau |^2\,d\tau
  +\int_0^t|y_\tau -y'_\tau |^2\,d\tau\right),\quad t\in[0,T].$$
Once this is proved, the required estimate follows immediately from
Gromwall's inequality.

Let $\phi$ be the function associated with $\O$ (see part 2 of Definition
\ref{admissible}).  For any constant $\gamma$, we have that \begin{align*}
  e&^{-\gamma[\phi(z_t)+\phi(z'_t)]}d\bigl(e^{\gamma[\phi(z_t)+\phi(z'_t)]}|z_t-z'_t|^2\bigr)
\\ =&2(z_t-z'_t)\cdot
\bigl[\bigl(\sigma(y_t)dw_t+d\ell_t\bigr)-\bigl(\sigma(y'_t)dw_t
+d\ell'_t\bigr)\bigr]\\ &+|z_t-z'_t|^2\gamma\bigl[\nabla\phi(z_t)
\cdot\bigl(\sigma(y_t)dw_t+d\ell_t\bigr)
+\nabla\phi(z'_t)\cdot\bigl(\sigma(y'_t)dw_t+d\ell'_t\bigr)\bigr]
\\ =&\bigl[(2(z_t-z'_t)+\gamma|z_t-z'_t|^2\nabla\phi(z_t))\cdot
\nu(z_t)\bigr]d|\ell|_t\\ &+\bigl[(2(z'_t-z_t)+\gamma|z_t-z'_t|^2\nabla\phi(z'_t))\cdot
\nu(z'_t)\bigr]d|\ell'|_t\\ &+\bigl[2(z_t-z'_t)\cdot\bigl(\sigma(y_t)-\sigma(y'_t)\bigr)
+\gamma|z_t-z'_t|^2\bigl(\nabla\phi(z_t)\sigma(y_t)+\nabla\phi(z'_t)\sigma(y'_t)\bigr)\bigr]dw_t
\end{align*}
Taking $\gamma=\frac{-2C_0}{\alpha}$, we have that (cf. Part 1
of Definition \ref{admissible}) the first two terms are less than or equal
to $0$.  Since $\sigma$ and $\nabla\phi$ are Lipschitz continuous and
$\frac{dw}{dt}$ is bounded on finite intervals, we know that there exists a
$C=C_w(T)<\infty $ such that
$$|z_t-z'_t|^2\leq
C\left(\int_0^t|z_\tau -z'_\tau |^2d\tau +\int_0^t|z_\tau -z'_\tau ||y_\tau -x'_\tau |d\tau
+\int_0^t|y_\tau -y'_\tau |^2d\tau \right)$$
for $t\in[0,R]$.  Thus, because
$|z_\tau -z'_\tau ||y_\tau -y'_\tau |\leq
\frac{1}{2}|z_\tau -z'_\tau |^2+\frac{1}{2}|y_\tau -y'_\tau |^2$, we get
our estimate after replacing $C$ by $2C$. \end{proof}

Once we have Lemma
\ref{contraclem}, one can apply a standard Picard iteration argument to
show that $F_w$ has a unique fixed point and that this fixed point is the
first component of the one and only pair $(x_t,\ell_t)$ which solves
(\ref{reflectedODE}).

We now want to describe a couple of important properties of the solution
$(x_\cdot,\ell_\cdot)$.  

\begin{lem} \label{variationinequality}  Let $(x_\cdot,\ell_\cdot)$ be the
  solution to (\ref{reflectedODE}) for a given input $w_\cdot$ and starting
  point $x_0\in\bar{\O}$.  
  Then there exists a constant $C$,
  depending only on $\sigma$, $b$, and $\O$, such that $$ d|x|_t\leq Cd|w|_t
  $$ \end{lem}

\begin{proof}
Set $y_t=x_0+\int_0^t\sigma(x_s)dw_s$.  Then, $x_\cdot=\Gamma(y_\cdot)$,
and so it follows from Theorem 2.2 in~\cite{LS} that
$d|\ell|_t\leq d|y|_t$.  
Since $\sigma$ is bounded on $\bar{\O}$, there exists a $C<\infty $ such that
$d|y|_t\leq Cd|w|_t$, and therefore, because $x_t=y_t+\ell_t$, we have that
$d|x|_t\leq d|y|_t+d|\ell|_t \leq C\bigl(d|w|_t+d|w|_t\bigr)$, 
from which the lemma follows immediately. \end{proof}

We now introduce a more geometric representation of the equation 
(\ref{reflectedODE}). For a closed set $D\subseteq\R^d$ and $z\in\R^d$, let
$d_{D}(z)\equiv \inf_{y\in D}|y-z|$ denote the distance from $z$ to $D$
and denote by
$$ T_{D}(z)\equiv \{v\in\R^d:\liminf_{h\searrow
  0}\frac{d_{D}(z+hv)}{h}=0\} $$
the \emph{tangent cone} (a.k.a. the contingent
cone) to $D$ at $z$. Finally, let $\text{{\rm proj}}_D(z)$ denote the (possibly multi-valued)
projection of $z$ onto $D$.

The following is a version of a representation
result which was introduced originally in~\cite{Cornet}.

\begin{thm}
\label{RepresentationTheorem}
Let $\O$ be a bounded, admissible set and $w_\cdot$ a fixed, piecewise
smooth input. If
$(x_\cdot,\ell_\cdot)$ is the unique solution to (\ref{reflectedODE}), then 
\begin{equation}
\label{geometricrep}
\dot{x}_t=\text{{\rm proj}}_{T_{\bar{\O}}(x_t)}(\sigma(x_t)\dot{w}_t),\quad t\text{-a.e.}
\end{equation}
Conversely, given a solution $x_\cdot$ to (\ref{geometricrep}), there exists an
$\ell_\cdot$ such that $(x_\cdot,\ell_\cdot)$ is a solution to (\ref{reflectedODE}). 
\end{thm}

\begin{rk} In general, the tangent cone $T_D(z)$ is only closed and not
  necessarily convex. However, Part 1 of Definition \ref{admissible}
  guarantees that $T_{\bar{\O}}(z)$ is convex for all $z\in\bar{\O}$
  (cf. Lemma \ref{convexfacts} below) and so $\text{{\rm proj}}_K(\cdot)$ is
  single valued.  \end{rk}

In order to prove Theorem \ref{RepresentationTheorem}, we will need to
introduce some concepts from convex analysis.  For more information
about these concepts and their properties, we refer the reader to the texts~\cite{RW}
and~\cite{Vinter}.

A non-empty set $K\subseteq\R^d$ is called a \emph{cone} if
$v\in K\Longrightarrow \lambda v\in K$ for all $\lambda \ge0$.   
Given a cone $K$, we denote by $K^*$ its \emph{polar cone} 
$K^*$ to be the set $\{w:v\cdot w\leq 0,\;\forall v\in K\}$.  Next, for a
given closed set $D\subseteq \R^d$ and a $z\in D$, we define the
\emph{proximal normal cone} to $D$ at $z$ to be the set
$$N^p_{D}(z)\equiv \{v\in \R^d:\exists C>0 \text{ s.t. }  (y-z)\cdot\nu\leq
C|z-y|^2,\quad \forall y\in D\}$$ 
and the \emph{Clarke tangent cone} to $D$ at $z$ to be the set
$$
\hat{T}_D(z)\equiv \{v\in\R^d: \forall z_n\in D \text{ s.t. } z_n\rightarrow
z,\quad\exists v_n\in T_D(z_n)\text{ s.t. }v_n\rightarrow v\}.$$
Note that $\hat{T}_D(z)$ is always convex.

%When the set $D$ is sleek, some of these cones are the same:
%\begin{thm}
%For a closed set $D\subseteq \R^d$, the following are equivalent:
%\begin{enumerate}
%\item $N_D(z)=\hat{N}_D(z)$
%\item $T_D(z)=\hat{T}_D(z)$
%\end{enumerate}
%When $D$ satisfies either of the above conditions we say that $D$ is \emph{sleek} (a.k.a. tangentially regular, Clarke regular).
%\end{thm}
%\begin{proof}
%See Corollary 6.29 in~\cite{RW}
%\end{proof}

We now present a lemma which records the properties of an admissible set
$\O$ in terms of these concepts.

\begin{lem}
\label{convexfacts}
Let $\O$ be admissible. Then
\begin{enumerate}
\item For each $z\in\partial\O$, 
$$
v\in N^p_{\bar{\O}}(z) \Longleftrightarrow \frac{-v}{|v|}\in\nu(z)\text{ for }v\neq 0
$$
\item The graph of $z\longrightarrow N^p_{\bar{\O}}(z)$ is closed. That is,
  if $z_i\in\bar{\O}$, $v_i\in N^p_{\bar{\O}}(z_i), z_i\rightarrow z,$ and
  $v_i\rightarrow v$, then $v\in N^p_{\bar{\O}}(z)$. 

\item $T_{\bar{\O}}(z)=\hat{T}_{\bar{\O}}(z)$, and so it is convex for all $z\in\bar{\O}$.

\item $N^p_{\bar{\O}}(z)=T_{\bar{\O}}(z)^*$ for all $z\in\bar{\O}$.

\end{enumerate}
\end{lem}

\begin{proof}
1. is immediate from our definitions.

2. follows from 1. and Part 1 of Definition \ref{admissible}. Indeed, there
exists a $C_0>0$ such that for each $i$, 
\begin{equation}
\label{starstar}
(z_i-y)\cdot v_i +C_0|v_i||z_i-y|^2\geq 0,\quad\forall y\in\bar{\O}
\end{equation}
(note that when $z_i\in\O$, $v_i=0$ and (\ref{starstar}) holds trivially).
Taking $i\rightarrow \infty$ we see that $
(z-y)\cdot v +C_0|v||z-y|^2\geq 0$ for all $y\in\bar{\O},$
from which it follows that $v\in N^p_{\bar{\O}}(z)$.

3. and 4. follow in a standard way from 2.
See Chapter 4. of~\cite{Vinter} and Chapter 6 of~\cite{RW} (in particular
Corollary 6.29) for the details. 
\end{proof}

Using ideas from~\cite{Cornet}, we now prove Theorem \ref{RepresentationTheorem}.
\begin{proof} (Proof of Theorem \ref{RepresentationTheorem}) First suppose
  $(x_\cdot,\ell_\cdot)$ is a solution to (\ref{reflectedODE}). From Theorem
  \ref{variationinequality} and its proof, we see that $x_\cdot$ and $\ell_\cdot$
  are locally Lipschitz and therefore that
$\dot{x}_t=\sigma(x_t)\dot{w}_t+\dot{\ell}_t,\;t\text{-a.e.}$
Since $x_{t+h}$ and $x_{t-h}$ are in $\bar{\O}$, we have that
\begin{equation}
\label{TandminusT}
\dot{x}_t\in -T_{\bar{\O}}(x_t)\cap T_{\bar{\O}}(x_t),\quad t\text{-a.e.},
\end{equation}
and, because $T_{\bar{\O}}(x_t)$ is convex,
$\dot{x}_t$ is the projection of $\sigma(x_t)\dot{w}_t$ onto $T_{\bar{\O}}(x_t)$ if and only if
$\bigl(\sigma(x_t)\dot{w}_t-\dot{x}_t)\cdot (v-\dot{x}_t\bigr)\leq 0$ for
all, $v\in T_{\bar{\O}}(x_t)$.  
Note that by property 1. of Lemma \ref{convexfacts}, $-\dot{\ell}_t\in N^p_{\bar{\O}}(x_t)$
(when $x_t\in\O$ this holds trivially), and so, by property 4. of Lemma
\ref{convexfacts} and (\ref{TandminusT}), we have that
$$
\dot{x}_t\cdot\dot{\ell}_t\leq 0,\quad \dot{x}_t\cdot\dot{\ell}_t\geq
0\Longrightarrow \dot{x}_t\cdot\dot{\ell}_t = 0.$$
Therefore, using property 4. again, we have that
$$
(\sigma(x_t)\dot{w}_t-\dot{x}_t)\cdot (v-\dot{x}_t)=-\dot{\ell}_t\cdot
(v-\dot{x}_t)=-\dot{\ell}_t\cdot v\leq 0 $$
as desired.

Conversely, suppose $x_\cdot$ is a solution to (\ref{geometricrep}), and set
$\ell_t\equiv \int_0^t \dot{x}_s-\sigma(x_s)\dot{w}_s ds$.  Then $\ell_0=0$ and, since
$\sigma$ is bounded, $\ell_\cdot$ is a continuous function of locally
bounded variation.  Finally, because $\dot{x}_t$ is the projection of
$\sigma(x_t)\dot{w}_t$ onto the convex set $T_{\bar{\O}}(x_t)$, we have that
$$
-\dot{\ell}_t\cdot (v-\dot{x}_t)=(\sigma(x_t)\dot{w}_t-\dot{x}_t)\cdot
(v-\dot{x}_t)\leq 0,\quad\forall v\in T_{\bar{\O}}(x_t).$$
Since $\dot{x}_t \in T_{\bar{\O}}(x_t)$ and $T_{\bar{\O}}(x_t)$ is a convex cone, for each
$v\in T_{\bar{\O}}(x_t)$, $x_t+v\in T_{\bar{\O}}(x_t)$.  Thus, by replacing $v$ with $v+x_t$
in the inequality above, we find that
$-\dot{\ell}_t\cdot v\leq 0$ for all $v\in T_{\bar{\O}}(x_t)$, 
and so $-\dot{\ell}_t\in T_{\bar{\O}}(x_t)^*= N^p_{\bar{\O}}(x_t)$.  
Finally, by property 1. of Lemma \ref{convexfacts},
this implies that $(x_t,\ell_t)$ is a solution to (\ref{reflectedODE}).
\end{proof}

%%%%%NEWSECTION%%%%%%%%%%%%%%%%%%%%%%%%%%

\section{\label{Step2} Tightness of the Approximating Measures}

Let $\bigl(C([0,\infty );\R^r),\F,\Wi\bigr)$ be the standard
  $r$-dimensional Wiener space.  That is, $\F$ is the Borel field for
$C([0,\infty );\R^r)$ and $\Wi$ is the standard Wiener measure.
We will use $W_\cdot$ to denote a generic
Wiener path and $\F_t$ to denote the $\sigma $-algebra generated by
$W_\cdot\restriction [0,t]$.  
Finally, for each positive integer $N$, let $W^N_{\cdot}$ denote the $N$-dyadic
linear polygonalization of $W_\cdot$.  That is, $W^N_{m2^{-N}}=W_{m2^{-N}}$
and $W^N_\cdot$ is linear on $[m2^{-N},(m+1)2^{-N}]$ for each $m\in\Bbb N$.

Next, $\O\subset\R^d$ will be a bounded, admissible domain, and
$b:\bar\O\longrightarrow \R^d$ and $\sigma:\bar\O\longrightarrow
\rm{Hom}(\R^r;\R^d)\bigr)$ will be uniformly Lipschitz continuous functions.
Given a starting point $x_0\in\bar{\O}$, for each $W^N_{\cdot}$,
$(X^N_\cdot,L^N_\cdot)$ will denote
the solution to the reflected ODE (\ref{reflectedODE}) with $w_t$ and
$\sigma (x)$ replaced by, respectively,
$$\begin{pmatrix} W^N_t\\ t \end{pmatrix}\in\R^r\times \R\quad\text{and} \quad
\begin{pmatrix}\sigma(x)\\b(x)\end{pmatrix}\in\rm{Hom}\bigl(\R^r\times \R;\R^d\times \R).$$ 
$\{X^N_t:\,t\ge0\}$ and $\{L^N_t:\,t\ge0\}$ are then progressively
  measurable with respect to $\{\F_t:\,t\ge0\}$, and we will use $\Bbb P^N$
  on the \emph{$(X,L,W)$-pathspace} 
$C\bigl([0,\infty );\bar\O\bigr)\times C\bigl([0,\infty );\R^d\bigr)\times
C\bigl([0,\infty );\R^r\bigr)$ to denote the distribution of the triple
  $(X^N_t,L^N_t,W^N_t)$ under $\Wi$.   

In first subsection, we show that the family $\{\Bbb P^N:\,N\ge0\}$ is
tight on the $(X,L,W)$-pathspace.  In second subsection, we also develop
some estimates which will needed for the next section. 

\subsection{Tightness of the ${\Bbb P}^N$}

By Kolmogorov's Continuity Criterion, we will know that $\{\Bbb
P^N:\,N\ge0\}$ is tight as soon as we prove that for each $m\in{\Bbb N}$
and $T>0$   there exists a $C_m(T)<\infty $, which is independent of $N$, such that
\begin{eqnarray}
\label{inductionW}
\Bbb E\bigl[|W^N_t-W^N_s|^{2^{m+1}}\bigr]&\leq& C_m(T)(t-s)^{2^m}
\end{eqnarray}
\begin{eqnarray}
\label{inductionX}
\Bbb E\bigl[|X^N_t-X^N_s|^{2^{m+1}}\bigr]&\leq& C_m(T)(t-s)^{2^m}
\end{eqnarray}
\begin{eqnarray}
\label{inductionL}
\Bbb E\bigl[|L^N_t-L^N_s|^{2^{m+1}}\bigr]&\leq& C_m(T)(t-s)^{2^m}.
\end{eqnarray}

First note that (\ref{inductionW}) is an easy consequence of the equality
$\Bbb E\bigl[|W_t-W_s|^{2^{m+1}}\bigr]=C_m(t-s)^{2^m}$ where $C_m=\Bbb
E\bigl[|W_1|^{2^{m+1}}\bigr]$.

The proofs of (\ref{inductionX}) and (\ref{inductionL}) are a little more
involved. 

\begin{lem}
There is a $C<\infty $ such that for all $s<t\le s+2^{-N}$,  
\begin{eqnarray}
\label{BVbound}
|X^N_t-X^N_s|\leq C|W^N_t-W^N_s|+C(t-s)
\end{eqnarray}
\end{lem}

\begin{proof}  When $s$ and $t$ lie in the same $N$-dyadic interval, 
this follows more or less immediately from Theorem \ref{variationinequality}.
Namely, 
$$\begin{aligned}
|X^N_t-X^N_s|&\leq |X^N|_t-|X^N|_s\leq C\bigl((|W^N|_t-|W^N|_s)+(t-s)\bigr)
\\&=C\bigl(|W^N_t-W^N_s|+(t-s)\bigr),\end{aligned}$$
where the last equality comes from the fact that $s$ and $t$ lie in the
same $N$-dyadic interval.  When they are in adjacent $N$-dyadic intervals,
one can reduce to the case when they are in the same $N$-dyadic interval by
an application of Minkowski's inequality.
\end{proof}

It remains to handle $s$ and $t$ with $t-s>2^{-N}$, and for this we will
need the next two lemmas.  Here, and elsewhere, $\floor{u}$ is shorthand
for the largest $N$-dyadic number $m2^{-N}$ dominated by $u$.  That is,
$\floor{u}$ equals $2^{-N}$ times the integer part of $2^Nu$.

\begin{lem}
For $m\geq 0$ there exists a $C_m<\infty $ such that for all $s<t$
\begin{eqnarray}
\label{crazylemma}
\Bbb E\left[\left(\int_s^t|W^N_u-W^N_\floor{u}|\ d|W^N|_u\right)^{2^m}\right]\leq C_m(t-s)^{2^m}
\end{eqnarray}
and
\begin{eqnarray}
\label{crazylemma2}
\Bbb E\left[\left(\int_s^t(u-\floor{u})\ d|W^N|_u\right)^{2^m}\right]\leq C_m(t-s)^{2^m}
\end{eqnarray}
\end{lem}

\begin{proof}
If $s<t$ lie in the same $N$-dyadic interval we have that
\begin{align*}
\int_s^t|W^N_u-W^N_\floor{u}|d|W^N|_u=&4^N|\Delta
W^N_\floor{s}|^2\int_s^t(u-\floor{u})du\leq 2^N|\Delta
W^N_\floor{s}|^2(t-s)\\ 
\int_s^t(u-\floor{u})d|W^N|_u=&2^N|\Delta
W^N_\floor{s}|\int_s^t(u-\floor{u})du\leq |\Delta W^N_\floor{s}|(t-s) 
\end{align*}
and so
\begin{equation}
\label{triangleMinkowski}
\begin{aligned}
&\Bbb E\left[\left(\int_s^t|W^N_u-W^N_\floor{u}|d|W^N|_u\right)^{2^m}\right]^{2^{-m}}\leq C_m(t-s)\\
&\Bbb E\left[\left(\int_s^t(u-\floor{u})d|W^N|_u\right)^{2^m}\right]^{2^{-m}}
\leq C_m2^{-N2^{m-1}}(t-s)\leq C_m(t-s).
\end{aligned}
\end{equation}
Applying the Minkowski inequality, we see that the inequalities
(\ref{triangleMinkowski}) continue to hold for general $s<t$. 
\end{proof}

\begin{lem}  Let $\phi$ and $\alpha $ be as in Part 2 of Definition \ref{admissible}, and set
  $\gamma =-\frac{2C_0}{\alpha }$, where $C_0$ is the constant in Part 1 of
  that definition.  
Given $s\ge0$, there exist $\{\F_t:\,t\ge0\}$ progressively measurable
functions $\{Z_{\tau ,s}:\,\tau \ge s\}$ and $\{V_{\tau ,s}:\,\tau \ge s\}$
satisfying 
\begin{equation}
\label{ZVProperties}
|Z^N_{u,s}|\leq C|X^N_u-X^N_s|,\text{ } |Z^N_{u_2,s}-Z^N_{u_1,s}|\leq
C|X^N_{u_2}-X^N_{u_1}|,\text{ and } |V^N_{u,s}|\leq C, 
\end{equation}
with a constant $C<\infty $, which is independent of $s$ and $N$, such that
\begin{equation}
\label{awesomeinequality}
e^{\gamma \phi(X^N_t)}
|X^N_t-X^N_s|^2\leq\int_s^t Z^N_{u,s} dW^N_u+\int_s^t V^N_{u,s} du\text{
  for all }t>s.
\end{equation}
\end{lem}

\begin{proof}  Just as in the proof of Lemma \ref{contraclem},
\begin{align*}
d\bigl(e^{\gamma\phi(X^N_t)}&|X^N_t-X^N_s|^2\bigr)\\
\leq&e^{\gamma\phi(X^N_t)}\Bigl(2(X^N_t-X^N_s)+\gamma|X^N_t-X^N_s|^2\nabla
\phi(X^N_t)\Bigr)\sigma(X^N_t)dW^N_t\\
&+e^{\gamma\phi(X^N_t)}\Bigl(2(X^N_t-X^N_s)+\gamma|X^N_t-X^N_s|^2\nabla\phi(X^N_t)\Bigr)b(X^N_t)dt.
\end{align*}
from which (\ref{awesomeinequality}) follows with
$$
Z^N_{u,s}=e^{\gamma\phi(X^N_u)}\Bigl(2(X^N_u-X^N_s)
+|X^N_u-X^N_s|^2\gamma\nabla\phi(X^N_u)\Bigr)\sigma(X^N_u) 
$$
and
$$
V^N_{u,s}=e^{\gamma\phi(X^N_u)}\Bigl(2(X^N_u-X^N_s)
+|X^N_u-X^N_s|^2\gamma\nabla\phi(X^N_u)\Bigr)\cdot b(X^N_u).$$
Since $\nabla\phi$, $b$, and $\sigma$ are Lipschitz continuous functions on
the bounded domain $\O$, it is clear how to choose the $C$ in (\ref{ZVProperties}).
\end{proof}
%%%%%%%%%%%%%%%%%

We now prove (\ref{inductionX}) in the case that $t-s>2^{-N}$ by
induction on $m$. Taking into account the fact that $\phi$ is bounded, we
can use (\ref{awesomeinequality}) to derive the estimate 
\begin{equation}
\label{inductionfun}
\begin{aligned}
\Bbb E\bigl[|X^N_t-X^N_s|^{2^{m+1}}\bigr]\leq&
C_m\Bbb E\left[\left(\int_s^t(Z^N_{u,s}-Z^N_{\floor{u},s})dW^N_u\right)^{2^m}\right]\\ 
&+ C_m\Bbb E\left[\left(\int_s^t Z^N_{\floor{u},s} dW^N_u\right)^{2^m}\right]\\&\quad+
C_m\Bbb E\left[\left(\int_s^t V^N_{u,s} dW^N_u\right)^{2^m}\right]
\end{aligned}
\end{equation}
for some $C_m<\infty $.  Because $V^N_{u,s}$ is bounded (see (\ref{ZVProperties})),
the third term is bounded by a constant times $(t-s)^{2^m}$. For the first
term we have that, for some constants $C<\infty $,
\begin{align*}
\Bbb E&\left[\left(\int_s^t(Z^N_{u,s}-Z^N_{\floor{u},s})dW^N_u\right)^{2^m}\right]
\leq C\Bbb E\left[\left(\int_s^t|X^N_u-X^N_\floor{u}|d|W^N|_u\right)^{2^m}\right]
\\&\leq C\Bbb E\left[\left(\int_s^t|W^N_u-W^N_\floor{u}|d|W^N|_u\right)^{2^m}\right]
+C\Bbb E\left[\left(\int_s^t(u-\floor{u})d|W^N|_u\right)^{2^m}\right]\\&\leq C(t-s)^{2^m},
\end{align*}
where the first inequality follows from (\ref{ZVProperties}), the second inequality
from (\ref{BVbound}), and the third inequality from (\ref{crazylemma}) and (\ref{crazylemma2}).
Finally, for the second term we have that
\begin{align*}
\Bbb E&\left[\left(\int_s^t Z^N_{\floor{u},s} dW^N_u\right)^{2^m}\right]
\leq C\Bbb E\left[\left(\int_s^t |Z^N_{\floor{u},s}|^2 du\right)^{2^{m-1}}\right]\\&\hskip.25truein
\leq C(t-s)^{(2^{m-1}-1)}\Bbb E\left[\int_s^t |Z^N_{\floor{u},s}|^{2^m}du\right]
\\&\hskip.25truein\leq C(t-s)^{(2^{m-1}-1)}\Bbb E\left[\int_s^t |X^N_\floor{u}-X^N_s|^{2^m}du\right]\\
&\hskip.25truein\leq C(t-s)^{(2^{m-1}-1)}\Bbb E\left[\int_s^t (\floor{u}-s)^{2^{m-1}}du\right]
\\&\hskip.25truein\leq C(t-s)^{(2^{m-1}-1)}\Bbb E\left[\int_s^t (t-s)^{2^{m-1}}du\right]
\leq C(t-s)^{2^m},
\end{align*}
where the first inequality is an application of Burkholder's inequality,
the third inequality follows from (\ref{ZVProperties}), the fourth inequality
is our induction hypothesis, and the fifth inequality follows from our assumption
that $t-s>2^{-N}$.  Hence we will be done once we show that
(\ref{inductionX}) holds when $m=0$.  But we can handle the base case by the
same estimates as above, only now noting that the second term of (\ref{inductionfun})
is $0$ in this case.  

Finally, we must prove (\ref{inductionL}).  Since
$$
dX^N_t=\sigma(X^N_t)dW^N_t+b(X^N_t)dt+dL^N_t
$$
we have that
\begin{align*}
\Bbb E\bigl[|L^N_t-L^N_s|^{2^{m+1}}\bigr]\leq &C
\Bbb E\bigl[|X^N_t-X^N_s|^{2^{m+1}}\bigr]+C\Bbb \Bbb E\left[\left
(\int_s^t\sigma(X^N_u)dW^N_u\right)^{2^{m+1}}\right]
\\&+C\Bbb E\left[\left(\int_s^t b(X^N_u)du\right)^{2^{m+1}}\right]\end{align*}
We already know that the first term is bounded from above by $C(t-s)^{2^{m}}$.
Moreover, because $b$ is bounded and $0\le s<t\le T$, the third term is bounded
above by a constant depending on $T$ times $(t-s)^{2^{m}}$. For the second
term we have that 
\begin{align*}
\Bbb E&\left[\left(\int_s^t\sigma(X^N_u)dW^N_u\right)^{2^{m+1}}\right]\leq
C\Bbb
E\left[\left(\int_s^t\bigl(\sigma(X^N_u)-\sigma(X^N_\floor{u}\bigr)dW^N_u\right)^{2^{m+1}}\right] 
\\&\hskip2.5truein+C\Bbb E\left[\left(\int_s^t\sigma(X^N_\floor{u})dW^N_u\right)^{2^{m+1}}\right]
\\&\leq C\Bbb
E\left[\left(\int_s^t|X^N_u-X^N_\floor{u}|d|W^N|_u\right)^{2^{m+1}}\right]+C(t-s)^{2^{m}}
\\&\leq C\left(\Bbb E\left[\left(\int_s^t|W^N_u-W^N_\floor{u}|d|W^N|_u\right)^{2^m}\right]
+\Bbb E\left[\left(\int_s^t(u-\floor{u})d|W^N|_u\right)^{2^m}\right]\right)
\\&\hskip.25truein+C(t-s)^{2^{m}}\leq C(t-s)^{2^{m-1}}
\end{align*}
where the second inequality follows is an application of Burkholder's inequality
and the fact that $\sigma$ is bounded, the third inequality follows from (\ref{BVbound}),
and the last inequality follows from (\ref{crazylemma}) and
(\ref{crazylemma2}).  Putting these inequalities together we get (\ref{inductionL}).

Given a $\psi :[0,\infty )\longrightarrow \R^d$, $\beta \in(0,1]$, and
$t>s>0$, set
$$\|\psi \|_{\beta,[s,t]}=\sup_{s\le u_1<u_2\le t}\frac{|\psi (u_2)-\psi
(u_1)|}{(u_2-u_1)^\beta }.$$ 
As an immediate consequence of the estimates in (\ref{inductionW}),
(\ref{inductionX}), and (\ref{inductionL}) combined with Kolmogorov's
Continuity Criterion (cf.\ Theorem 3.1.4 in~/cite{StroockBook}), we have the following theorem.

\begin{thm}
\label{Holder}
For each $\beta <\frac12$, $p\in(1,\infty )$, and $T>0$, there exists a $K_{\beta ,p}(T)<\infty $
such that
$$\P\bigl(\|W^N\|_{\beta ,[0,T]}\vee\|X^N\|_{\beta ,[0,T]}\vee\|L^N\|_{\beta
  ,[0,T]}\ge R\bigr)\le K_{\beta ,p}(T)R^{-p}\text{ for  }R>0.$$
\end{thm}

%%%%%%%%%%%%%%%%%%%%%%%%%%%%%%%%%%%%%%%%%%%%%%%%%%%%%%%%%%%%%%%%%%%%%%%%%%%%%%%
%%%%%%%%%%%%%%%%%%%%%%%%%%%%%%%%%NEWSECTION%%%%%%%%%%%%%%%%%%%%%%%%%%%%%%%%%%%%%%%
%%%%%%%%%%%%%%%%%%%%%%%%%%%%%%%%%%%%%%%%%%%%

\subsection{Controlling the Variation of $L_\cdot^N$}

In general, the variation of a function cannot be controlled by its uniform
norm.  Thus, before we can apply the tightness result in the previous
subsection to get the sort of result which we are seeking, we must give a
separate argument which shows that the 
variation of $L_\cdot^N$ can be estimated in terms of its uniform norm.  To
be precise, Theorem \ref{thm:mrlemma} says that the variation of
$L^N_\cdot\restriction [0,t]$ 
can be estimated in terms of the uniform norm of $L^N_\cdot\restriction [0,t]$
and the H\"older norm of $X^N_\cdot\restriction [0,t]$.  Hence, since
Theorem \ref{Holder} provides control on the H\"older, and therefore the uniform,
norms of the three processes $W^N_\cdot$, $X^N_\cdot$, and $L^N_\cdot$, our
tightness result will sufficient for our purposes (cf.\ Theorems
\ref{TheoremMartingale} below).  

In the following, and elsewhere, $\|\psi \|_{[t_1,t_2]}=\sup_{\tau \in[t_1,t_2]}|\psi (\tau )|$.

\begin{thm}
\label{thm:mrlemma}
For all $0\leq s<t$,
\begin{eqnarray}
\label{mrlemma}
|L^N|_t-|L^N|_s\leq C\bigl((t-s)R^{-4}\|X^N\|_{\frac14,[s,t]}^4+1\bigr)\|L^N\|_{[s,t]},
\end{eqnarray}
where $R$ is the constant given in Part 3 of Definition \ref{admissible}.
\end{thm}

Our proof follows the proof of Lemma 1.2 in~\cite{LS}.

\begin{proof}  Let $\O_1,\ldots,\O_n$ denote the open balls
  $B(x_1,2R),\ldots,B(x_n,2R)$ appearing in Part 3 of
  Definition \ref{admissible}, and choose an open set $\O_0$ so 
  such that $\bar{\O}_0\subseteq\O$ and $\bar{\O}\subseteq\O_0\cup \bigcup_{i=1}^n
  B(x_i,R)$.  Given $x\in\bar\O$, let $k(x)$ be the smallest $1\le k\le n$
  such that $x\in B(x_k,R)$, or otherwise let $k(x)$ be $0$.  Next, set $\zeta _0=s$ and define $\zeta _m$ for
  $m\ge1$ inductively so that
$$\zeta _{m+1}=t\wedge \inf\bigl\{\tau \ge\zeta _{m}:\,X^N_\tau\notin
  \O_{k(X^N_{\zeta _{m}})}\bigr\}.$$

Consider the time interval $[\zeta _m,\zeta _{m+1}]$.  If $\zeta _m<t$ and
$k(X^N_{\zeta _m})=0$, then $L^N_\cdot\restriction [\zeta _m,\zeta _{m+1}]$
is constant and so $|L^N|_{\zeta _{m+1}}-|L^N|_{\zeta _m}=0$.  If $\zeta
_m<t$ and $k_m\equiv k(X^N_{\zeta _m})\ge1$, then (cf.\ Part 3 of
Definition \ref{admissible})
$$
(L^N_{\zeta _{m+1}}-L^N_{\zeta _m})\cdot a_{k_m}=\int_{\zeta _m}^{\zeta _{m+1}}{\nu(X^N_\tau )\cdot a_{k_m}
d|L^N|_\tau }\geq \lambda\bigl(|L^N|_{\zeta _{m+1}}-|L^N|_{\zeta _m}\bigr).
$$
Hence, in either case, 
$$
|L^N|_{\zeta _{m+1}}-|L^N|_{\zeta _m}\leq C|L^N_{\zeta _{m+1}}-L^N_{\zeta
  _m}|\leq C\|L^N\|_{[s,t]}.  
$$
At the same time, if $\zeta _{m+1}<t$ and $k(X^N_{\zeta _m})\ge1$, then
$|X^N_{\zeta _{m+1}}-X^N_{\zeta _m}|\geq R$ and so
$$
\frac{R}{(\zeta _{m+1}-\zeta _m)^{\frac{1}{4}}}\leq \frac{|X^N_{\zeta
    _{m+1}}-X^N_{\zeta _m}|}{(\zeta _{m+1}- 
\zeta _m)^{\frac{1}{4}}}\leq \|X^N\|_{\frac14,[0,t]}.
$$
Thus if ${\mathcal M}=\sup\{m:\zeta _{m+1}<t\}$, then
$$
\frac {\mathcal M}2\leq 1+\bigl|\bigl\{m:\,\zeta _{m+1}<t\text{ and } k(X^N_{\zeta _m})\ge1\bigr\}\bigr|
\le1+\frac{(t-s)\|X^N\|_{\frac14,[s,t]}^4}{R^4},$$
which, in conjunction with the preceding, means that
\begin{align*}
|L^N|_t-|L^N|_s&\leq \sum_{m=0}^{{\mathcal M}-1} (|L^N|_{T_{m+1}}-|L^N|_{T_m})+ (|L^N|_t-|L^N|_{T_{\mathcal M}})\\
&\leq \bigl({C\mathcal M}+2\bigr)\|L^N\|_{[s,t]}
\leq C\bigl[(t-s)R^{-4}\|X^n\|_{\frac14,[s,t]}^4+1\bigr]\|L^N\|_{[s,t]}
\end{align*}
\end{proof}

\section{\label{Step3} Associated Martingale\\ and Submartingale Problems}

We know that the sequence of measures $\{\Bbb P^N:\,N\ge0\}$ is on $(X,L,W)$-pathspace.
Our eventual goal is to show that this sequence converges.  Equivalently,
we want to show that all limit points are the same.  In this section we
will show that every limit solves martingale and submartingale
problems, and in the next section we will show that this fact is sufficient
to check that convergence takes place.       

Up until now, we have needed only the assumptions that $\O$ is
bounded and admissible, and $\sigma$ and $b$ are Lipschitz continuous.
However, starting now, we will be assuming that 
$\sigma\in C^2\bigl(\bar{\O};\rm{Hom}(\R^r;\R^d)\bigr)$.  In addition, it
will be convenient to make a change in our notation.  Instead to writing
the equation which determines $(X^N_t,L^N_t)$ (pathwise) as 
\begin{equation}
\label{theequation}
dX^N_t=\sigma(X^N_t)dW^N_t+b(X^N_t)dt+dL^N_t,\quad X^N_0=x_0,
\end{equation}
we will use the equivalent expression 
\begin{equation}
\label{Hormander}
dX^N_t=\sum_{i=1}^r V_i(X^N_t)d(W^N_i)_t+V_0(X^N_t)dt+dL^N_t,\quad X^N_0=x_0
\end{equation}
where $V_i$ is the $i$th column of the matrix $\sigma$ and $V_0=b$.
At the same time, we introduce the vector fields $\tilde
V_i:\bar{\O}\longrightarrow \R^d\times \R^r$ given by $\tilde
V_i=\begin{pmatrix}V_i\\e_i\end{pmatrix}$ for $1\le i\le r$ and $\tilde
V_0=\begin{pmatrix} V_0\\0\end{pmatrix}$, where $\{e_1,\dots,e_r\}$ is the
standard, orthonormal basis in $\R^r$.  Then, $\P^N$-almost surely,
\begin{equation}
\label{Hormander2}
dY_t=\sum_{i=1}^r \tilde{V}_i(X_t)d(W_i)_t+\tilde{V}_0(X_t)dt,\quad Y_0=\left( \begin{array}{c}
x_0\\
0 \end{array} \right)
\end{equation}
where $Y_t=\left( \begin{array}{c}
X_t-L_t\\
W_t \end{array} \right)$.  In keeping with this notation, we use
$D_{V_i}$ and $D_{\tilde V_i}$ to denote the directional derivative
operators on $\R^d$ and $\R^d\times \R^r$ determined, respectively, by $V_i$
and $\tilde V_i$.  Finally, for $\xi \in \R^d$, $T_\xi$ will denote the
translation operator on $C\bigl(\R^d\times \R^r;\R\bigr)$ given by $T_\xi\varphi 
(x,y)=\varphi (x-\xi ,y)$.  

\begin{thm}
\label{TheoremMartingale}
Let $\P$ be any limit point of the sequence $\{{\Bbb P}^N:\,N\ge0\}$.
Then for all $h\in C^2_{\rm b}(\R^{d}\times \R^r;\R)$,
\begin{equation}
\label{MartingaleProblem}
h(Y_t)-\int_0^t\left(\frac{1}{2}\sum_{i=1}^r\bigl[D^2_{\tilde{V}_i}T_{L_s}h\bigr](X_s,W_s) +\bigl[
D_{\tilde{V}_0}T_{L_s}h\bigr](X_s,W_s)\right)ds
\end{equation}
is a $\P$-martingale relative to the filtration $\{{\mathcal B}_t:\,t\ge0\}$
generated by the paths in the $(X,L,W)$-pathspace.  
Also, for all $f\in C^2_{\rm b}(\R^d;\R)$ satisfying $\frac{\partial f}{\partial
  \nu}(x)\geq 0$ for every $x\in\partial\O$ and $\nu\in\nu(x)$,
\begin{equation}
\label{SubMartingaleProblem}
f(X_t)-f(x_0)-\int_0^t\left(\frac{1}{2}\sum_{i=1}^rD^2_{V_i}f(X_s) +D_{V_0}f(X_s)\right)ds
\end{equation}
is a $\P$-sub-martingale relative to the filtration $\{{\mathcal B}_t:\,t\ge0\}$.  
\end{thm}

We will begin with the proof of the martingale property for
(\ref{MartingaleProblem}), and, without loss in generality, we will do so
under the assumption that $h$ is smooth and compactly supported.  
What we need to show is that for any limit point $\P$, $0\le s<t$ and bounded, continuous,
${\mathcal B}_s$-measurable $F:C\bigl([0,\infty );\R^d\times\R^d\times 
  \R^r\bigr)\longrightarrow [0,\infty )$, 
\begin{equation}
\label{martingaleincrement}
\Bbb E^{\P}\left[\left(h(Y_t)-h(Y_s)-\int_s^t \tilde\L h(u) du\right)F\right]=0
\end{equation}
where we have used $\tilde \L h(u)$ to denote the integrand in
(\ref{MartingaleProblem}), and clearly it suffices to check this when $s$
and $t$ are $M$-dyadic rationals for some $M\in{\Bbb N}$.  Thus, it suffices to show that
\begin{equation}
\label{MartingaleGoal}
\bE^{{\Bbb P}^N}\left[\left(h(Y_t)-h(Y_s)-\int_s^t \tilde\L h(u) du\right)F\right]\rightarrow 0
\end{equation}
for $M$-dyadic $s$ and $t$ and bounded, ${\mathcal B}_s$-measurable
$F\in C\bigl([0,\infty );\R^d\times \R^d\times \R^r\bigr)$.  

For $N\ge M$, write
$$
h(Y_t)-h(Y_s)=\sum_{m=2^Ns}^{2^Nt-1} h(Y_{(m+1)2^{-N}})-h(Y_{m2^{-N}}),
$$
and, for each term in the sum, use (\ref{Hormander2}) to see that,
$\P^N$-almost surely,
\begin{align*}
h(Y_{(m+1)2^{-N}})-h(Y_{m2^{-N}})=&\int_{m2^{-N}}^{(m+1)2^{-N}}\sum_{i=1}^r
  \bigl[D_{\tilde{V}_i}T_{L_\tau }h\bigr](X_\tau,W_\tau )(\dot{W}_{i,m})d\tau
\\&\qquad+\int_{m2^{-N}}^{(m+1)2^{-N}}
  \bigl[D_{\tilde{V}_0}T_{L_\tau }h\bigr](X_\tau,W_\tau )d\tau, \end{align*}
where $\dot W_{i,m}\equiv 2^N\bigl(W_i((m+1)2^{-N})-W_i(m2^{-N})\bigr)$.  

Since 
$$\sum_{m=2^Ns}^{2^Nt}\int_{m2^{-N}}^{(m+1)2^{-N}}
  \bigl[D_{\tilde{V}_0}T_{L_\tau }h\bigr](X_\tau,W_\tau )d\tau=\int_s^t
  \bigl[D_{\tilde{V}_0}T_{L_\tau }h\bigr](X_\tau,W_\tau )d\tau,$$
the second term on the right causes no problem.

To handle the first term, note that
\begin{align*}
\bigl[D_{\tilde V_i}T_{L_\tau }h\bigr](X_\tau ,W_\tau )=&\bigl[D_{\tilde
    V_i}T_{L_{m2^{-N}}}h\bigr](X_\tau ,W_\tau ) \\&-\sum_{k=1}^d
\int_{m2^{-N}}^\tau\bigl[D_{\tilde V_i}T_{L_\sigma }\partial _{x_k}h\bigr](X_\tau ,W_\tau )\,dL_\sigma .
\end{align*}
Since the second term on the right is dominated by a constant times
$|L|_{(m+1)2^{-N}}-|L|_{m2^{-N}}$, we see that
\begin{align*}
\bE^{\P^N}&\left[\left|\sum_{m=2^Ns}^{2^Nt}\left(\int_{m2^{-N}}^{(m+1)2^{-N}}
\Bigl(\bigl[D_{\tilde V_i}T_{L_\tau }h\bigr]-\bigl[D_{\tilde
    V_i}T_{L_{m2^{-N}}}h\bigr]\Bigr)(X_\tau ,W_\tau )\,d\tau \right)\dot
W_{i,m}\right|\right]\\&\le C2^{-\frac
  N4}\bE^{\P^N}\bigl[|L|_t\|W\|_{\frac14,[0,t]}\bigr]\longrightarrow 0\end{align*}
as $N\to\infty $.

Next, use (\ref{Hormander}) to see that
\begin{align*}\bigl[D_{\tilde V_i}&T_{L_{m2^{-N}}}h\bigr](X_\tau,W_\tau )
=\bigl[D_{\tilde V_i}T_{L_{m2^{-N}}}h\bigr](X_{m2^{-N}},W_{m2^{-N}})\\&
\int_{m2^{-N}}^\tau \bigl[D_{\tilde V_0}D_{\tilde V_i}T_{L_{m2^{-N}}}h\bigr](X_\sigma ,W_\sigma )\,d\sigma
+\sum_{k=1}^d\bigl[\partial _{x_k}D_{\tilde V_i}T_{L_{m2^{-N}}}h\bigr](X_\sigma ,W_\sigma )\,dL_\sigma
\\&\hskip.5truein
+\sum_{j=1}^r\dot W_{j,m}\int_{m2^{-N}}^\tau \bigl[D_{\tilde V_j}D_{\tilde V_i}T_{L_{m2^{-N}}}h\bigr]
(X_\sigma ,W_\sigma )\,d\sigma .\end{align*}
Since the conditional $\P^N$-expected value of
$$\dot W_{i,m}\bigl[D_{\tilde V_i}T_{L_{m2^{-N}}}h\bigr](X_{m2^{-N}},W_{m2^{-N}})$$
given ${\mathcal B}_s$ is zero, the first term on the right does not appear in the
computation.  Moreover, 
After integrating the second two terms over
$[m2^{-N},(m+1)2^{-N}]$, multiplying by $\dot W_{i,m}$, and summing from
$m=2^N$ to $m=2^Nt$, one can
easily check that the absolute values of the resulting quantities have
$\P^N$-expected values which tend to $0$ as $N\to\infty $.

Finally, again applying (\ref{Hormander}), one finds that
$$\int_{m2^{-N}}^\tau \bigl[D_{\tilde V_j}D_{\tilde V_i}T_{L_{m2^{-N}}}h\bigr]
(X_\sigma ,W_\sigma )\,d\sigma $$
can be replaced by
$$(\tau -m2^{-N})\bigl[D_{\tilde V_j}D_{\tilde V_i}T_{L_{m2^{-N}}}h\bigr]
(X_{m2^{-N}},W_{m2^{-N}})$$
plus terms which make no contributions in the limit as $N\to\infty $.
Hence, we are left with quantities of the form
$$\sum_{m=2^Ns}^{2^Nt}2^{-2N-1}\dot W_{j,m}\dot W_{i,m}\bigl[D_{\tilde V_j}D_{\tilde V_i}T_{L_{m2^{-N}}}h\bigr]
(X_{m2^{-N}},W_{m2^{-N}}).$$
Since the $\P^N$-conditional expected value of $2^{-2N}\dot W_{j,m}\dot
W_{i,m}$ is $2^{-N}\delta _{i,j}$,
\begin{align*}\bE^{\P^N}&\left[\left(\sum_{m=2^Ns}^{2^Nt}2^{-2N-1}\dot W_{j,m}\dot
  W_{i,m}\bigl[D_{\tilde V_j}D_{\tilde V_i}T_{L_{m2^{-N}}}h\bigr]
(X_{m2^{-N}},W_{m2^{-N}})\right)F\right]\\&
=\frac{\delta _{i,j}}2\bE^{\P^N}\left[2^{-N}\left(\sum_{m=2^Ns}^{(m+1)2^N}
\bigl[D_{\tilde V_j}D_{\tilde V_i}T_{L_{m2^{-N}}}h\bigr]
(X_{m2^{-N}},W_{m2^{-N}})\right)F\right],\end{align*}
which, as $N\to\infty $, has that same limit as
$$\frac{\delta _{i,j}}2\bE^{\P^N}\left[\left(\int_s^t
\bigl[D_{\tilde V_j}D_{\tilde V_i}T_{L_s}h\bigr]
(X_{s},W_{s})\,ds\right)F\right].$$

The proof of (\ref{SubMartingaleProblem}) is similar, but easier, and so we will skip
the details.  The only difference is that when we apply (\ref{Hormander}) to 
the difference $f(X_{(m+1)2^{-N}})-f(X_{m2^{-N}})$, we throw away the
$dL_\tau $ integral since, under our hypotheses, it is non-negative.

%%%%%NEWSECTION

\section{\label{Step4} Convergence}

In this section we complete our program of proving the $\{\P^N:\,N\ge0\}$
converges to the distribution of an appropriate Stratonovich reflected 
SDE.  By the uniqueness result of Lions and Sznitman (Theorem 3.1
of~\cite{LS}) and the tightness which we proved in 3.1, the
convergence will follow as soon as we show that every limit $\P$ is
the distribution of that reflected SDE.

Let $\P$ be any limit of $\{{\Bbb P}^N:\,N\ge0\}$.  
By Theorem \ref{TheoremMartingale}, we know that, for all
$h\in C^2_{\rm b}(\R^{d}\times \R^r;\R)$, 
\begin{eqnarray} \label{martingale}
  h(X_t-L_t,W_t)-h(x_0,0)-\int_0^t\tilde\L h(s) ds\text{ is a }\P\text{
    martingale}, \end{eqnarray}
relative to $\{{\mathcal B}_t:\,t\ge0\}$, where
$$\tilde\L h(s)=\frac12\sum_{i=1}^r\bigl[D_{\tilde
    V_i}^2T_{L_s}h\bigr](X_s,W_s)+\bigl[D_{\tilde
    V_0}T_{L_s}h\bigr](X_s,W_s).$$
Using elementary stochastic calculus, it follows from (\ref{martingale})
that $\{W_t:\,t\ge0\}$ is a $\P$-Brownian motion relative to $\{{\mathcal
  B}_t:\,t\ge0\}$ and that, $\P$-almost surely, 
$$X_t-x_0-\int_0^t
  \left(\frac{1}{2}\sum_{i=1}^r [D_{V_i}V_i](X_s)+V_0(X_s) \right)ds -
  L_t=\int_0^t\sigma(X_s)dW_s, $$
which can be rewritten in Stratonovich form as
\begin{eqnarray}\label{SDE}
X_t-X_0=\sum_{i=1}^r\int_0^t V_i(X_s)\circ dW_s+\int_0^tV_0(X_s)\,ds+L_t.\end{eqnarray}
Thus, the only remaining question is whether $\{L_t:\,t\ge0\}$ has the
required properties.  That is, whether, $\P$-almost surely, $|L|_t<\infty$
and $\int_0^t{\mathbf 1}_\O(X_s)\,d|L|_s=0$ for all $t\ge0$, and
$\frac{dL_t}{d|L|_t}\in \nu (X_t)$ a.e.

Since the local variation norm is a lower semi-continuous function of local
uniform convergence, Theorem \ref{thm:mrlemma} tells us that, $\P$-almost
surely, $L_\cdot$ has locally bounded variation.  In fact, by combining that theorem 
with the estimates in Theorem \ref{Holder}, one sees that, for all $t\ge0$,
$|L|_t$ has finite $\P$-moments of all orders.  

In order to prove the other properties of $L_\cdot$ we will use the second part of Theorem
(\ref{TheoremMartingale}), which says that for every $f\in C^2_{\rm b}(\R^d;\R)$ satisfying
$\frac{\partial f}{\partial \nu}(x)\geq 0$ for all $x\in\partial\O$ and $
\nu\in\nu(x)$, 
\begin{eqnarray} \label{submartingale}
  f(X_t)-\int_0^t\L f(X_s) ds\text{ is a }\P\text{
    sub-martingale} \end{eqnarray}
relative to $\{{\mathcal B}_t:\,t\ge0\}$, where
$$\L f(x) =\frac12\sum_{i=1}^r D_{V_i}^2f(x)+D_{V_0}f(x).$$
Now compare this to what one gets by applying It\^o's formula to
(\ref{SDE}).  Namely, his formula says that if
$\xi ^f_t=\int_0^t\nabla f(X_s)\cdot dL_s$ then
$$f(X_t)-\int_0^t\L f(X_s)-\xi ^f_t\quad\text{is a $\P$-martingale}.$$
Thus, $\xi ^f_\cdot$ is $\P$-almost surely non-decreasing.  Starting from
this observation and using the arguments in Lemmas 2.3 and 2.5
of~\cite{SV}, one can prove the following lemma.

\begin{lem}
\label{xilem} For $f\in C^2_{\rm b}(\R^d;\R)$, define $\xi ^f_\cdot$ as
above.  Then, $\P$-almost surely, $\int_0^\infty {\mathbf
  1}_\O(X_s)\,d|\xi^f|_s=0$.  Moreover, if $\frac{\partial f}{\partial
  \nu}(x)\ge0$ for all $x$ in an open set $U$ and all $\nu \in\nu (x)$,
then, $\P$-almost surely, $t\rightsquigarrow\int_0^t{\mathbf 1}_U(X_s)\,d\xi
^f_s$ is non-decreasing.\end{lem}

Because $e_i\cdot L_\cdot=\xi ^{x_i}_\cdot$, it is obvious from the first
part of Lemma \ref{xilem} that $\int_0^\infty {\mathbf 1}_\O(X_s)\,d|L|_s=0$
$\P$-almost surely, and so all that we have to do is show that, $\P$-almost
surely, $\frac{dL_t}{d|L|_t}\in\nu (X_t)$ a.e.  To this end, let $\phi$ be
the function in Part 2 of Definition \ref{admissible}, and define
$$a^f(x)=\inf_{\nu \in\nu (x)}\frac{\frac{\partial f}{\partial \nu}(x)}{\frac{\partial \phi}
{\partial \nu}(x)}\text{ and }b^f(x)=
\sup_{\nu \in\nu (x)}\frac{\frac{\partial f}{\partial
    \nu}(x)}{\frac{\partial \phi}{\partial \nu}(x)}$$
for $x\in\partial \O$.

\begin{lem}\label{nuxlem}  If $\{x_n:\,n\ge1\}\subseteq\partial \O$,
  $\nu _n\in\nu (x_n)$ for each $n\ge1$, and $(x_n,\nu _n)\longrightarrow
  (x,\nu )$ in $\partial \O\times \Bbb S^{N-1}$, then $\nu \in\nu(x)$.  In
  particular, for each $f\in C^2_{\rm b}(\R^d;\R)$, $a^f$ is lower
  semicontinuous and $b^f$ is upper semicontinuous on $\partial \O$.  
  Furthermore, if $(x,\ell)\in\partial \O\times \Bbb S^{N-1}$ and there
exists a $\beta \ge0$ such that
$\nabla f(x)\cdot\ell\ge\beta a^f(x)$ for a set $S$ of $f\in C^2_{\rm
  b}(\R^d;\R)$ with the property that $\{\nabla f(x):\,f\in S\}$ is dense
in $\R^d$, then $\ell\in\nu (x)$.\end{lem}

\begin{proof}  The initial assertion is an easy consequence of Parts 1 and 2
  of Definition \ref{admissible}.  Next, suppose that
  $x_n\longrightarrow x$ in $\partial \O$.  Because, by the first
  assertion, $\nu (y)$ is compact for each $y\in\partial \O$, for each
  $n\ge1$ there is a $\nu _n\in\nu(x_n)$ such that $a_f(x_n)=\frac{\nabla f(x_n)\cdot\nu _n}
{\nabla \phi(x_n)\cdot\nu _n}$.  Now choose a subsequence
$\{x_{n_m}:\,m\ge1\}$ so that $\varliminf_{n\to\infty
}a^f(x_n)=\lim_{m\to\infty }a^f(x_{n_m})$ and $\nu _{n_m}\longrightarrow
\nu $ in $\Bbb S^{N-1}$.  Then
$\nu \in\nu(x)$ and so
$$a^f(x)\le\frac{\nabla f(x)\cdot\nu }
{\nabla \phi(x)\cdot\nu }\le\liminf_{n\to\infty }a^f(x_n).$$
The same argument shows that $b^f$ is upper semicontinuous.

Next, let
$(x,\ell)$ and $\beta $ be as in the final assertion.  Then, by
Part 2 of Definition \ref{admissible}.  By taking $f$ to be linear in a
neighborhood of $\bar\O$, one sees that for every $v\in\R^d$ there exists a
$\nu\in\nu(x)$ such that $v\cdot \ell\ge\beta \frac{v\cdot\nu}{\nabla
  \phi(x)\cdot\nu}$.  Hence, for each $x'\in\O$ there is a $\nu\in\nu(x)$
such that
$$(x'-x)\cdot\ell\ge\beta \frac{(x'-x)\cdot\nu }{\nabla\phi(x)\cdot\nu }\ge
-\frac{\beta C_0}\alpha |x'-x|^2,$$
which, by (\ref{elementaryalg}), means that $\ell\in\nu(x)$.
\end{proof}

\begin{lem} \label{helperlemma} For each $f\in C^2_{\rm b}(\R^d;\R)$,
  $\P$-almost surely $d\xi^f_\cdot$ is
  absolutely continuous with respect to $d\xi^\phi_\cdot$ and $a^f(X_t)\le\frac{d\xi ^f_\cdot}{d
\xi ^\phi_\cdot}(t)\le b^f(X_t)$ for $d\xi ^\phi_\cdot$-almost every $t\ge0$.
\end{lem}

\begin{proof}  First observe that $f\rightsquigarrow \xi ^f_\cdot$ is linear.
Now choose $\lambda >0$ so that $\nabla(\lambda \phi-f)(x)\cdot\nu\ge0$ for
all $x\in\partial \O$ and $\nu\in\nu(x)$.  Then, $\xi ^{\lambda
  \phi-f}_\cdot=\lambda \xi ^\phi_\cdot- \xi ^f_\cdot$ is
$\P$-almost surely non-decreasing, which proves that $d\xi ^f_\cdot\ll d\xi ^\phi_\cdot$ and
that $\frac{d\xi ^f_\cdot}{d\xi ^\phi_\cdot}\le\lambda $ $\P$-almost surely.

The proof that, $\P$-almost surely, $\alpha (t)\equiv\frac{d\xi ^f_\cdot}{d\xi ^\phi_\cdot}(t)$
lies between $a^f(X_t)$ and $b^f(X_t)$ for $d\xi ^\phi_\cdot$-almost every
$t\ge0$ is a simple localization of the preceding.  For example, to prove
the lower bound, use the lower semicontinuity of $a^f$ to choose, for each
$n\ge1$, a finite cover of $\partial \O$ by open balls
$B(x_{k,n},r_n),\;1\le k\le k_n$ such that 
$x_{k,n}\in\partial \O$, $r_n\le\frac1n$, and $a^f(y)\ge a^f(x_{k,n})-\frac
1n$ for all $y\in B(x_{k,n},r_n)\cap\partial \O$.  Then
$$\frac{\partial f}{\partial \nu}
(y)\ge\bigl(a^f(x_{k,n})-\tfrac1n\bigr)\frac{\partial \phi}{\partial \nu
}(y) \text{ for all }1\le k\le k_n,\;y\in B(x_{k,n},r_n),\text{ and } \nu
\in\nu(y).$$
Now let $\mu $ be the Borel measure on $C\bigl([0,\infty );\R^d\times
  \R^d\times \R^r\bigr)\times [0,\infty )$ determined by
$$\mu \bigl(\Gamma\times [a,b]\bigr)=\bE^\P\bigl[\xi ^\phi(b)-\xi
      ^\phi(a),\,\Gamma \bigr]$$
for all Borel subsets $\Gamma $ of $C\bigl([0,\infty );\R^d\times
  \R^d\times \R^r\bigr)$ and all $a<b$.  
Then, by Lemma \ref{xilem}, we can find a Borel measurable set $A\subseteq
C\bigl([0,\infty );\R^d\times \R^d\times \R^r\bigr)\times [0,\infty )$ 
whose complement has $\mu $-measure $0$ and on which both
$$X_\cdot\in\partial \O\text{ and }
{\mathbf 1}_{B(x_{k,n},r_n)}(X_\cdot)\bigl(a^f(x_{k,n})-\tfrac1n\bigr)\le
{\mathbf 1}_{B(x_{k,n},r_n)}(X_\cdot)\frac{d\xi ^f_\cdot}{d\xi ^\phi_\cdot}$$
hold for all $n\ge1$ and $1\le k\le k_n$.  Hence,
again by the lower semicontinuity of $a^f$, we see that $\frac{d\xi
  ^f_\cdot}{d\xi ^\phi_\cdot}\ge a(X_\cdot)$.
The proof of the upper bound is the same.\end{proof}

\begin{thm}\label{convergencethm}  Let $\P^N$ be the distribution of
  $(X^N,L^N)$ under Wiener measure.  Then $\{\P^N:\,N\ge0\}$ converges to
  the distribution $\P$ of the solution to the reflected stochastic
  differential equation (\ref{SDE}).\end{thm}

\begin{proof}  As we said earlier, everything comes down to showing that if
  $\P$ is a limit of $\{\P^N:\,N\ge0\}$ then, $\P$-almost surely
  $\ell_\cdot\equiv\frac{dL_\cdot}{d|L|_\cdot}\in\nu (X_\cdot)$
  $d|L|_\cdot$-almost everywhere.  Thus, because, without loss in generality, we
  may assume that $|\ell_\cdot|\equiv 1$, the second part of Lemma
  \ref{nuxlem} says that it suffices for us to show that, $\P$-almost surely, there exist a
  $\beta _\cdot\ge0$ such that $\nabla f(X_\cdot)\cdot\ell_\cdot\ge\beta _\cdot
  a^f(X_\cdot)$ $d|L|_\cdot$-a.e. for sufficiently many $f$'s.  
To this end, first note that, since $\xi^\phi_\cdot$ is $\P$-almost surely
non-decreasing, $\beta _\cdot\equiv\nabla \phi(X_\cdot)\cdot\ell_\cdot\ge0$
$d|L|_\cdot$-a.e. $\P$-almost surely.  Second, because
$L_\cdot=\sum_{i=1}^d\xi ^{x_i}_\cdot$ $\P$-almost surely, we know that,
$\P$-almost surely, $d|L|_\cdot\ll d\xi ^\phi_\cdot$ and that, for each
$f\in C^2_{\rm b}(\R^d;\R)$, 
\begin{equation}\nabla f(X_\cdot)\cdot\ell_\cdot=\frac{d\xi
  ^f_\cdot}{d|L|_\cdot}=\frac{d\xi ^f_\cdot}{d\xi ^\phi_\cdot}\nabla\phi(X_\cdot)\cdot\ell_\cdot
\ge \beta _\cdot a^f(X_\cdot)\quad d|L|_\cdot\text{-a.e.}\tag{*}\end{equation}
Finally, let $D$ be a countable, dense subset of $\R^d$, and for each $v\in D$
choose $f_v\in C^2_{\rm b}(\R^d;\R)$ so that $f_v(x)=v\cdot x$ in a
neighborhood of $\bar\O$.  Then, $\P$-almost surely, (*) holds
simultaneously with $f=f_v$ for every $v\in D$.
\end{proof}

\begin{rk}  In our derivation of Theorem (\ref{convergencethm}) we used
  (\ref{submartingale}) to show that $L_\cdot$ has the required properties.
  However, using the ideas in Lemma 1.3 of~\cite{LS}, we could have based
  our proof on the fact that the approximating $L^N_\cdot$'s had these
  properties.  Our choice of proof was dictated by two considerations.
  First, it seemed to us to be the simpler one.  Second, and more
  important, it brings up an interesting question.  Namely, does
  (\ref{submartingale}) by itself determine $\P$?  In~\cite{SV} it was
  shown that (\ref{submartingale}) determines $\P$ when $\O$ has a smooth
  boundary and $\L$ is strictly elliptic, even if the coefficients are not
  smooth.  Thus, the question is whether the same result holds when $\O$ is
  only admissible and the coefficients of $\L$ are smooth but may be
  degenerate. 
\end{rk}

%%%%NEWSECTION

\section{\label{OandA}Observations and Applications}

It should be noticed that although the approximating $L^N_\cdot$'s as well as
limit $L_\cdot$ have locally bounded variation, the we cannot replace our
$(X,L,W)$-pathspace with one in which the middle component is the space
of continuous paths of locally bounded variation.  The reason is that
although $L^N_\cdot$ will be absolutely continuous, $L_\cdot$ will not.
Indeed, consider reflected Brownian motion on the halfline $[0,\infty )$.
  In this case $L^N_t=\sup_{0\leq s\leq t}[-W^N_s]$ is piecewise constant
  and therefore absolutely continuous.  On the other hand, $L_t=\sup_{0\leq
    s\leq t}[-W_s]$, which is the local time at $0$ of $W_\cdot$ and as
  such is singular.  

The main application of our result that we consider is the following:
Suppose that for each $N$, the paths $X^N_t$ satisfy a certain geometric
property almost surely and the set $S$ of paths which satisfy this
geometric property is closed in $C([0,\infty);\R^d)$. It then follows that
  the paths of $X_t$ also satisfy this geometric property almost surely
  since

\begin{equation}
\label{limitargument}
\P(S)\geq \limsup_{N\rightarrow\infty}{\P}^N(S)=1
\end{equation}

where, abusing notation, we use $\P^N$ and $\P$ to denote the marginal distributions of
${\P}^N$ and $\P$ on $X$-pathspace. That is, $\P^N(A)={\P}^N(A\times
C([0,\infty);\R^d)\times C([0,\infty);\R^r))$ and $\P(A)=\P(A\times
    C([0,\infty);\R^d)\times C([0,\infty);\R^r))$.  We conclude with several
        examples of the sort of application which we have in mind.

\begin{example}
In $\R^2$, let $\O$ be the rectangle $[-1,1]\times [0,2]$. Fix
$x_0\in\bar{\O}$ and consider the Stratonovich reflected SDE $$
dX_t=\sigma(X_t)\circ dW_t+ dL_t,\quad X_0=x_0, $$ where $\sigma(x)=\left(
\begin{array}{c} x_2\\ -x_1 \end{array} \right)$.  Then
\begin{equation}
  \label{larger} \text{If }|x_0|>1,\quad |X_t|\leq |x_0|\text{ for
  }t>0\quad \P\text{-a.s.}  \end{equation}
and
\begin{equation}
  \label{same} \text{If }|x_0|<1,\quad |X_t|= |x_0|\text{ for }t>0\quad
  \P\text{-a.s.}  \end{equation} \end{example}

\begin{proof} In view of (\ref{limitargument}), it suffices to prove that
  (\ref{larger}) and (\ref{same}) hold $\P^N$-a.s. The distribution of
  $X_{\cdot}$ under $\P^N$ is, in view of Theorem
  \ref{RepresentationTheorem}, the same as the distribution of
  $X^N_{\cdot}$ under $\Wi$, where $X^N_{\cdot}$ solves the ODE $$
  \dot{X}^N_t=\text{{\rm
      proj}}_{T_{\bar{\O}}(X^N_t)}(\sigma(X^N_t)\dot{W}^N_t),\quad
  X^N_0=x_0 $$ It is easy to check that $\forall x\in\bar{\O},a\in\R$,
  $x\cdot\text{{\rm proj}}_{T_{\bar{\O}(x)}}(\sigma(x)a)$ is non-negative
  or non-positive according as $|x|\geq 1$ or $|x|\leq 1$.  Hence, because,
  for each $W_t$, $\frac{d}{dt}\left(|X^N_t|^2\right)=2X^N_t\cdot\text{{\rm
      proj}}_{T_{\bar{\O}(X^N_t)}}(\sigma(X^N_t)\dot{W}^N_t)$ $dt$-a.e.,
  (\ref{larger}) and (\ref{same}) for $X^N_\cdot$ are obvious.  Figure
  \ref{Applications2} shows a sample path of $X^N_t$ under $\Wi$ (to save
  space, we denote the ``intended velocity'' $\sigma(X^N_t)\dot{W}^N_t$ by
  $v_t$ and the ``actual velocity'' $\text{{\rm
      proj}}_{T_{\bar{\O}}(X^N_t)}(\sigma(X^N_t)\dot{W}^N_t)$ by
  $\tilde{v}_t$).

%%%%%%%%%%%%%%%%%%%%%%%%%%%%%%%%%%%%%%%%%%%%%%%%%%%%%%%%
\begin{figure}[!htbp]
\centering
\includegraphics{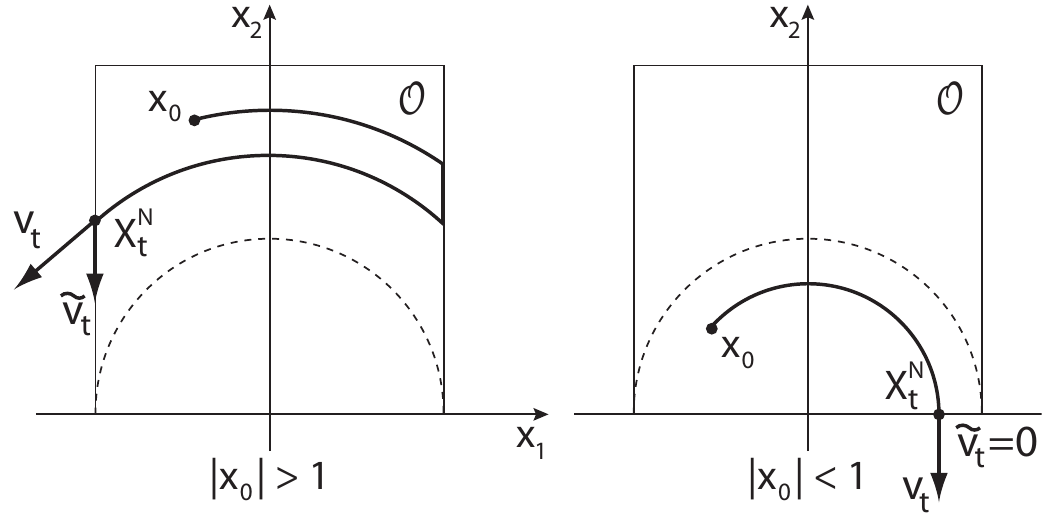}
\caption{\label{Applications2}}
\end{figure}
%%%%%%%%%%%%%%%%%%%%%%%%%%%%%%%%%%%%%%%%%%%%%%%%%%%%%%%
\end{proof}

We next consider coupled reflected Brownian motion, for which we will need
the following lemmas.

\begin{lem} \label{productdomain} Suppose $\O$ is bounded and
  admissible. Then $\O\times\O$ is bounded and admissible as
  well. Furthermore, for each $(x,y)\in\partial(\O\times\O)$, the set of
  normal vectors $\nu(x,y)$ defined by (\ref{defnofnu}) has the
  representation
\begin{multline}
\label{productrep}
\nu(x,y)=\left\{\left( \begin{array}{c}
a_1\nu_x\\
a_2\nu_y \end{array} \right): \nu_x\in\nu(x), \nu_y\in\nu(y),
a_1^2+a_2^2=1, a_1,a_2>0\right\},\\ 
\text{ when }(x,y)\in\partial\O\times\partial\O,
\end{multline}
$$\nu(x,y)=\left\{\left( \begin{array}{c}
\nu_x\\
0 \end{array} \right): \nu_x\in\nu(x)\right\},\text{ when }(x,y)\in\partial\O\times\O,
$$
and
$$
\nu(x,y)=\left\{\left( \begin{array}{c}
0\\
\nu_y \end{array} \right): \nu_y\in\nu(y)\right\},\text{ when }(x,y)\in\O\times\partial\O.
$$
\end{lem}

\begin{proof} The representation formulae are a straightforward consequence of
  the definition of inward pointing unit proximal normal vectors in
(\ref{defnofnu}). That $\O\times\O$ satisfies Part 1 of
  Definition \ref{admissible} follows from the representation formulae and
  the fact that $\O$ satisfies Part 1 of Definition \ref{admissible}.

We next show that $\O\times\O$ satisfies Part 2 of Definition
\ref{admissible}. Since $\O$ is bounded, $\phi$ is bounded in $\O$ and so
after adding a constant to $\phi$ if necessary, we may assume that
$\phi\geq 1$ in $\bar{\O}$.

Let $\Phi(x,y)\equiv \phi(x)\phi(y)$. Then for all
$(x,y)\in\partial(\O\times\O),\quad \nu\in\nu(x,y)$, we have, by our
representation formulae, that
\begin{align*}
\nabla\Phi(x,y)\cdot \nu =& a_1\phi(y)\nabla\phi(x)\cdot\nu_x +
a_2\phi(x)\nabla\phi(y)\cdot\nu_y\\ 
                   \geq& a_1 \phi(y)\alpha + a_2\phi(x)\alpha
                   \geq \alpha(a_1+a_2)\geq \alpha
\end{align*}
(where $(a_1,a_2)=(1,0)$ and $(0,1)$ for the cases $(x,y)\in(\partial\O\times\O)\cup(
\O\times\partial\O)$), and so Part 2 holds with the function
$\Phi(x,y)$. Finally, as $\O\times\O$ is bounded, Part 3 follows
immediately from Part 2. 
\end{proof}

\begin{lem}
\label{productlemma}
Let $\O$ be bounded and admissible. Then for $(x,y)\in\bar{\O}\times\bar{\O}$,
$$
T_{\bar{\O}\times\bar{\O}}(x,y)=T_{\bar{\O}}(x)\times T_{\bar{\O}}(y).
$$
Furthermore,
$$
\text{{\rm proj}}_{T_{\bar{\O}\times\bar{\O}}(x,y)}\left( \begin{array}{c}
\xi\\
\eta \end{array} \right)=\left( \begin{array}{c}
\text{{\rm proj}}_{T_{\bar{\O}}(x)}(\xi)\\
\text{{\rm proj}}_{T_{\bar{\O}}(x)}(\eta) \end{array} \right)
$$
\end{lem}

\begin{proof}
When $D\subset\R^d$ is admissible, it follows from Part 3. of Lemma \ref{convexfacts} that
$$
T_{\bar{D}}(z)=\{v\in\R^d:\lim_{h\searrow 0}\frac{d_{\bar{D}}(z+hv)}{h}=0\}
$$
(i.e. $\lim$ replaces $\liminf$). Since $\O$, and by Lemma \ref{productdomain}, $\O\times\O$, are bounded and admissible, the first statement then follows immediately from the relation
$$
d^2_{\bar{\O}\times\bar{\O}}\left(\left( \begin{array}{c}
x\\
y \end{array} \right)+h\left( \begin{array}{c}
v\\
w \end{array} \right) \right)=d^2_{\bar{\O}}(x+hv)+d^2_{\bar{\O}}(y+hw).
$$
The second statement then follows from the first by a similar argument.
\end{proof}

\subsection{Synchronously Coupled Reflected Brownian Motion}

We now discuss synchronously coupled reflected Brownian motion. A
$d$-dimen\-sional synchronously coupled reflected Brownian motion is a
$2d$-dimensional process $Z_t=(X_t,Y_t)$ in a product domain
$\bar{\O}\times\bar{\O}$ which satisfies the reflected SDE
$$
dZ_t=\sigma(Z_t)dW_t+dL_t,
$$
where
$$
\sigma(z)\equiv \left( \begin{array}{c}
I\\
I \end{array} \right).
$$
Note that, because $\sigma$ is constant, there is no difference between the
Stratonovich and It\^o versions of the above SDE.  We will express this
reflected SDE in a more convenient form as the pair of reflected SDEs 
$$
dX_t=dW_t+dL_t, \quad X_0=x_0\quad\text{and}\quad dY_t=dW_t+dM_t, \quad Y_0=y_0.
$$

We think of $X_t$ and $Y_t$ as being two $d$-dimensional processes which
are driven by the same Brownian motion $W_t$ and which are constrained to
lie in the same domain $\bar{\O}$. The two processes move in sync except
for when one or the other is bumps against the boundary and gets nudged.

We now consider the geometric properties of synchronously coupled reflected
Brownian motion in two domains. Such properties were used to prove the
``hot spots conjecture'' for these domains (See~\cite{HotSpots} and
~\cite{HotSpotsLip} for more details).

\begin{example} \label{triangleexample} Let $\O\subset \R^2$ be the obtuse
  triangle lying with its longest face on the horizontal axis, and denote
  its left and right acute angles by $\alpha$ and $\beta$. Suppose $x_0\neq
  y_0$, and for $x\neq y$, let $\angle(x,y)=\arg(y-x)$. Then, $\P$-almost
  surely,
\begin{equation} \label{angledance} -\beta\leq
    \angle(x_0,y_0)\leq\alpha\implies \text{ for all $t$ either }
    -\beta\leq \angle(X_t,Y_t)\leq\alpha \text{ or } X_t=Y_t.
\end{equation} \end{example}

\begin{proof} By (\ref{limitargument}), it suffices to show that
  (\ref{angledance}) holds $\P^N$-a.s. Fix $N$ and $W_t\in \Omega$. In view
  of Theorem \ref{RepresentationTheorem} and Lemma \ref{productlemma} it
  will suffice to show that $X^N_t$ and $Y^N_t$ satisfy (\ref{angledance})
  where $X^N_t$ and $Y^N_t$ satisfy the ODE \begin{equation}
    \label{approxprojection} \begin{split} \dot{X}^N_t=&\text{{\rm
          proj}}_{T_{\bar{\O}}(X^N_t)}(\dot{W}^N_t),\quad
      X^N_0=x_0\\ \dot{Y}^N_t=&\text{{\rm
          proj}}_{T_{\bar{\O}}(Y^N_t)}(\dot{W}^N_t),\quad Y^N_0=y_0
  \end{split} \end{equation} It is straightforward to check that the
  functions $X^N_t$, $Y^N_t$ starting at $X^N_0=x_0$, $Y^N_0=y_0$ and
  defined inductively for $t\in[m2^{-N},(m+1)2^{-N}]$ by $X^N_t=\text{{\rm
      proj}}_{\bar{\O}}(X^N_{m2^{-N}}+(t-m2^{-N})\dot{W}^N_t)$ and
  $Y^N_t=\text{{\rm
      proj}}_{\bar{\O}}(Y^N_{m2^{-N}}+(t-m2^{-N})\dot{W}^N_t)$ satisfy
  (\ref{approxprojection}). A simple geometric argument shows that if
  $\angle (x,y)\in[-\beta,\alpha]$ then $\angle (\text{{\rm
      proj}}_{\bar{\O}}(x),\text{{\rm
      proj}}_{\bar{\O}}(y))\in[-\beta,\alpha]$ or $\text{{\rm
      proj}}_{\bar{\O}}(x)=\text{{\rm proj}}_{\bar{\O}}(y)$. From this it
  follows by induction that $X^N_t$ and $Y^N_t$ satisfy (\ref{angledance})
  as desired.
Figure \ref{Applications3} shows a pair of sample paths $X^N_t$ and $Y^N_t$
in the interval $m2^{-N}\leq t\leq (m+1)2^{-N}$ where we use $v$ to denote
the constant vector $2^N(W_{(m+1)2^{-N}}-W_{m2^{-N}})$.
\end{proof}
%%%%%%%%%%%%%%%%%%%%%%%%%%%%%%%%%%%%%%%%%%%%%%%%%%%%%%%%
\begin{figure}[!htbp] \centering \includegraphics{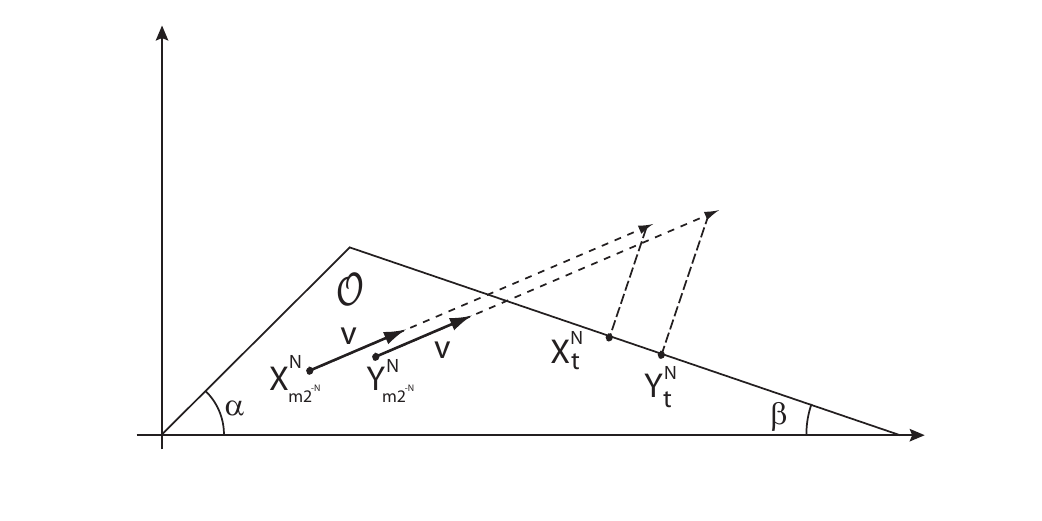}
  \caption{\label{Applications3}} \end{figure}
%%%%%%%%%%%%%%%%%%%%%%%%%%%%%%%%%%%%%%%%%%%%%%%%%%%%%%% \end{proof}

\begin{example} \label{lipexample} (Proposition 2 in ~\cite{HotSpotsLip})We
  now consider synchronously coupled reflected Brownian motion in a Lip
  domain. A lip domain is a domain in $\R^2$ which is bounded
  below by a function $f_1(x)$ and above by another function $f_2(x)$ each
  of which is Lipschitz continuous with constant bounded by $1$. The
  domains are so named because they look like a pair of lips (See Figure
  \ref{Applications5}).

%%%%%%%%%%%%%%%%%%%%%%%%%%%%%%%%%%%%%%%%%%%%%%%%%%%%%%%%
  \begin{figure}[!htbp] \centering \includegraphics{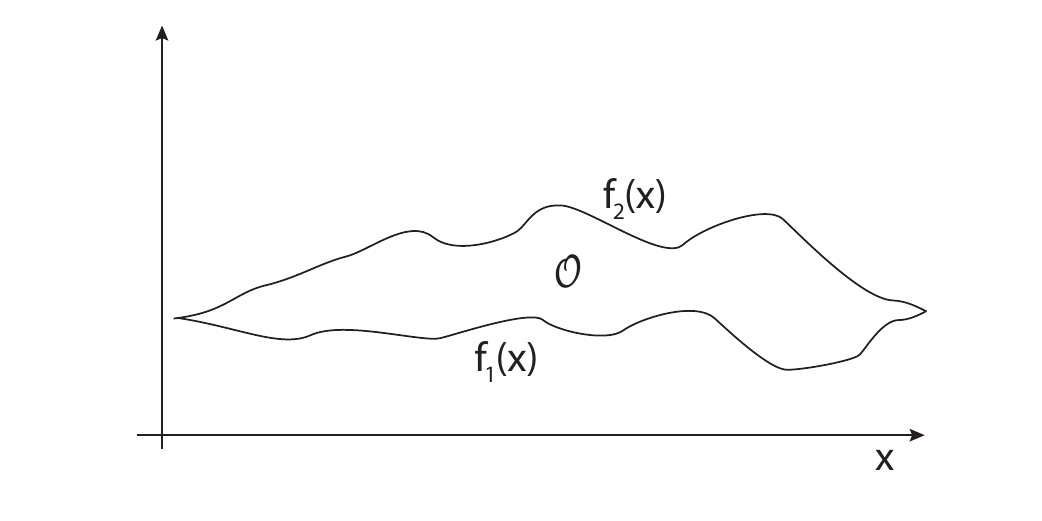}
    \caption{\label{Applications5}} \end{figure}
  %%%%%%%%%%%%%%%%%%%%%%%%%%%%%%%%%%%%%%%%%%%%%%%%%%%%%%%

Consider synchronously coupled reflected Brownian motion in a lip domain
$\O$ where the defining functions $f_1(x)$ and $f_2(x)$ are smooth and have
Lipschitz constants bounded by $\lambda<1$. Then $\O$ is a bounded
admissible domain. Recall the definition of $\angle(x,y)$ from the previous
example, and let $x_0,y_0\in\R^2$ be such that $x_0\neq y_0$ and
$\angle(x_0,y_0)\in[-\frac{\pi}{4},\frac{\pi}{4}]$. We have the following
geometric property for the paths $X_t$ and $Y_t$: \begin{equation}
  \label{angledance2} \forall t,\text{ either }
  \angle(X_t,Y_t)\in[-\frac{\pi}{4},\frac{\pi}{4}] \text{ or } X_t=Y_t\quad
  \P\text{-a.s.}  \end{equation} \end{example}

\begin{proof} \label{ProofLip} In view of (\ref{limitargument}) and Theorem
  \ref{RepresentationTheorem} it suffices to show that for every
  $W_t\in\Omega$, $X^N_t$ and $Y^N_t$ satisfy (\ref{angledance2}) when
  $X^N_t$ and $Y^N_t$ solve the ODE (\ref{approxprojection}). Let
  $\Theta^N_t\equiv\angle(X^N_t,Y^N_t)$ or $0$ according to whether
  $X^N_t\neq Y^N_t$ or $X^N_t=Y^N_t$. 
It is enough to show that, $dt$-almost everywhere,
  $\dot{\Theta}^N_t\leq 0$ when
  $\Theta^N_t\in[\frac{\pi}{4},\frac{\pi}{2}-\tan^{-1}(\lambda)]$ and
  $\dot{\Theta}^N_t\geq 0$ when
  $-\Theta^N_t\in[\frac{\pi}{4},\frac{\pi}{2}-\tan^{-1}(\lambda)]$. By
  symmetry it will suffice to prove the first statement.

Let $v_t=\dot{W}^N_t$, $\tilde{v}_t=\text{{\rm
    proj}}_{T_{\O}(X^N_t)}(v_t)$, and $\tilde{v}'_t=\text{{\rm
    proj}}_{T_{\O}(Y^N_t)}(v_t)$. We compute:
\begin{align*}
  \frac{d}{dt}\left[\Theta^N_t\right]=&\frac{d}{dt}\tan^{-1}
\left(\frac{(Y^N_t-X^N_t)_2}{(Y^N_t-X^N_t)_1}\right)= 
  \frac{(Y^N_t-X^N_t)\cdot
    R(\tilde{v}_t-\tilde{v}'_t)}{|Y^N_t-X^N_t|^2}\\ =&
  \frac{(Y^N_t-X^N_t)\cdot
    R(\tilde{v}_t-v_t)}{|Y^N_t-X^N_t|^2}+\frac{(Y^N_t-X^N_t)\cdot
    R(v_t-\tilde{v}'_t)}{|Y^N_t-X^N_t|^2}\\ \end{align*} where $R=\left(
\begin{array}{cc} 0 & -1\\ 1 & 0 \end{array} \right)$ is the matrix which
rotates vectors in $\R^2$ by $90^\circ$ counter-clockwise. Suppose
$\Theta^N_t\in [\frac{\pi}{4},\frac{\pi}{2}-\tan^{-1}(\lambda)]$. Then
since the Lipschitz constants of $f_1$ and $f_2$ are strictly less than
$1$, $X^N_t$ cannot be on the $f_2$-boundary and $Y^N_t$ cannot be on the
$f_1$-boundary. For each $t$, it follows that either $v_t=\tilde{v}_t$ or
$\arg(R(\tilde{v}_t-v_t))\in[\pi-\tan^{-1}(\lambda),\pi+\tan^{-1}(\lambda)]$
and either $v_t=\tilde{v}'_t$ or
$\arg(R(v_t-\tilde{v}'_t))\in[\pi-\tan^{-1}(\lambda),\pi+\tan^{-1}(\lambda)]$. And
so each of the terms in the sum above is $\leq 0$.
We depict in Figure \ref{Applications6} the case where $X^N_t\in\O$ and
$Y^N_t\in\partial\O$.

%%%%%%%%%%%%%%%%%%%%%%%%%%%%%%%%%%%%%%%%%%%%%%%%%%%%%%%%
\begin{figure}[!htbp] \centering \includegraphics{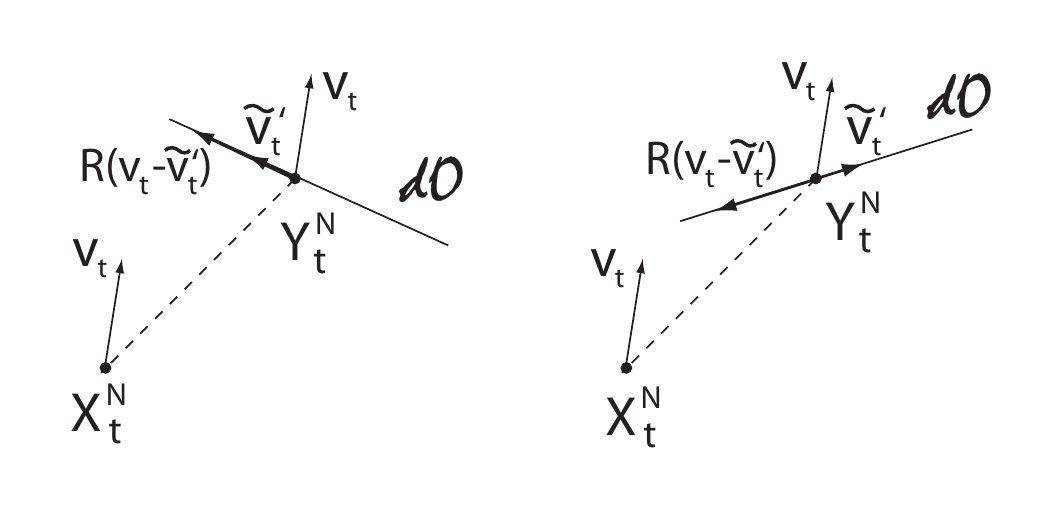}
  \caption{\label{Applications6}} \end{figure}
%%%%%%%%%%%%%%%%%%%%%%%%%%%%%%%%%%%%%%%%%%%%%%%%%%%%%%%

\end{proof}

\subsection{Mirror Coupled Reflected Brownian Motion}

Our final example involves mirror coupled reflected Brownian motion. A
$d$-dimensional mirror coupled reflected Brownian motion is a
$2d$-dimensional process $Z_t=(X_t,Y_t)$ in a product domain
$\bar{\O}\times\bar{\O}$ which satisfies the reflected SDE
\begin{equation}
  \label{mirrorRSDE} dZ_t=\sigma(Z_t)dW_t+dL_t, \end{equation} where $$
\sigma(z)=\sigma(x,y)\equiv \left( \begin{array}{c}
  I\\ I-2\frac{(y-x)(y-x)^\top}{|y-x|^2} \end{array} \right), $$
defined up
until the first time $\tau$ that $Z_t$ hits the diagonal of
$\bar{\O}\times\bar{\O}$, at which point we stop our process
(i.e. $Z_t\equiv Z_\tau$ for $t\geq \tau$).  We will express this reflected SDE
in a more convenient form as the pair of reflected SDEs
\begin{equation}
  \label{mirror2} \begin{split} dX_t=&dW_t+dL_t,\quad 
    X_0=x_0\\ dY_t=&(I-2\frac{(Y_t-X_t)(Y_t-X_t)^\top}{|Y_t-X_t|^2})dW_t+dM_t,\quad 
    Y_0=y_0.  \end{split} \end{equation}
We think of $X_t$ and $Y_t$ as
being two $d$-dimensional processes which are ``mirror coupled'' with
respect to the driving Brownian motion $W_t$ and which are constrained to
lie in the same domain $\bar{\O}$. That is, if you consider the hyperplane
which perpendicularly bisects the line segment connecting $X_t$ and $Y_t$
to be a ``mirror'', then the two processes move in such a way that they are
mirror images of each other until either process bumps into the boundary
and is nudged (which causes the mirror to shift). We refer the reader to
the papers~\cite{HotSpots} and~\cite{HotSpotsLip} for a more thorough
overview.

We will prove the same geometric property we considered for synchronously
coupled reflected Brownian motion in Example \ref{lipexample}, but now for
mirror coupled reflected Brownian motion. The point is that
(\ref{mirrorRSDE}) can be viewed as a Stratonovich reflected SDE and so
again it suffices to prove the geometric property for the approximating
processes.

We make this rigorous with the following lemma which shows that, off of the
diagonal of $\O\times\O$, the Stratonovich correction factor for
(\ref{mirrorRSDE}) is $0$.  \begin{lem} \label{lem:ItoisStrat} For
  $t<\tau$, \begin{equation} \label{ItoisStrat} \sum_{j=1}^d \frac{1}{2}d\left<
\left(I-2\frac{(Y_t-X_t)(Y_t-X_t)^\top}{|Y_t-X_t|^2}\right)_{ij},(W_t)_j \right>=0
  \end{equation}
In fact,
\begin{equation} \label{ItoisStrat2} d\left<\left(
I-2\frac{(Y_t-X_t)(Y_t-X_t)^\top}{|Y_t-X_t|^2}\right)_{ij},(W_t)_j
    \right>=0,\text{ for each } j \end{equation} \end{lem}

\begin{proof} It suffices to prove (\ref{ItoisStrat2}). Let $V_i
  \equiv (Y_t-X_t)_i$, where we have suppressed the dependence of $V_i$ on
  $t$. An easy calculation shows that $$ d\langle
  V_\ell,(W_t)_j\rangle=\frac{-2V_jV_\ell}{\sum_k V_k^2}dt $$ and
  \begin{align*} \frac{\partial}{\partial V_i}\left(\frac{V_iV_j}{\sum_k
      V_k^2}\right)=&\frac{V_j(\sum_k V_k^2)-2V_i^2V_j}{(\sum_k
      V_k^2)^2}\\ \frac{\partial}{\partial V_j}\left(\frac{V_iV_j}{\sum_k
      V_k^2}\right)=&\frac{V_i(\sum_k V_k^2)-2V_iV_j^2}{(\sum_k
      V_k^2)^2}\\ \frac{\partial}{\partial
      V_\ell}\left(\frac{V_iV_j}{\sum_k
      V_k^2}\right)=&\frac{-2V_iV_jV_\ell}{(\sum_k V_k^2)^2},\text{ for
    }\ell\neq i,j\\ \end{align*} Putting these together, we have that
  \begin{align*} d\langle \frac{V_iV_j}{\sum_k V_k^2},(W_t)_j \rangle
    =&\left(\frac{V_j(\sum_k V_k^2)-2V_i^2V_j}{(\sum_k V_k^2)^2}
    \right)\left(\frac{-2V_jV_i}{\sum_k
      V_k^2}\right)\\ &+\left(\frac{V_i(\sum_k V_k^2)-2V_iV_j^2}{(\sum_k
      V_k^2)^2} \right)\left(\frac{-2V_j^2}{\sum_k
      V_k^2}\right)\\ &+\sum_{\ell\neq
      i,j}\left(\frac{-2V_iV_jV_\ell}{(\sum_k V_k^2)^2}
    \right)\left(\frac{-2V_jV_\ell}{\sum_k V_k^2}\right)= 0
  \end{align*} From this, (\ref{ItoisStrat2}) immediately follows.

\end{proof}

We now prove a geometric property.

\begin{example} (Example \ref{lipexample} for mirror coupling) Let $\O$ be
  the same lip domain defined by smooth functions considered in Example
  \ref{lipexample} and consider the mirror coupled reflected Brownian
  motion starting from $x_0$ and $y_0$ where $x_0\neq y_0$. Then
  (\ref{angledance2}) holds where where $X_t$ and $Y_t$ are given by
  (\ref{mirror2}).  \end{example}

\begin{proof} Let $D_\epsilon
  =\{z=(x,y)\in\bar{\O}\times\bar{\O}:|x-y|<\epsilon\}$ be the
  ``$\epsilon$-diagonal'' of $\bar{\O}\times\bar{\O}$. Consider a sequence
  of smooth functions $\rho_k:\bar{\O}\times\bar{\O}\rightarrow [0,1]$ such
  that $\rho(z)\equiv 0$ on $D_{\frac{1}{2k}}$ and $\rho(z)\equiv 1$ off of
  $D_{\frac{1}{k}}$. Let $\sigma_k(z)=\rho_k(z)\sigma(z)$. Then
  $\sigma_k\in C^2(\bar{\O}\times\bar{\O})$.

Let $\P^k$ be the measure on $Z$-pathspace induced by the solutions to the
reflected SDE $$ dZ^k_t=\sigma_k(Z^k_t)\circ dW_t+dL^k_t $$ and define
$\P^{k,N}$ to be the measures on $Z$-pathspace induced by solutions to the
approximating reflected ODE \begin{equation} \label{Nkapprox}
  dZ^{k,N}_t=\sigma_k(Z^{k,N}_t)dW^N_t+dL^{k,N}_t \end{equation} Recall
that the stopping time $\tau$ corresponds to the first time $X_t$ equals
$Y_t$ and define $\tau_k \equiv  \inf \{t:|X_t-Y_t|<\frac{1}{k}\}$. Let
$S=\{Z_t\in C([0,\infty);\R^{2d}):
  -\frac{\pi}{4}\leq\angle(X_t,Y_t)\leq\frac{\pi}{4},\forall t<\tau\}$ and
  let $S_k=\{Z_t\in C([0,\infty);\R^{2d}):
    -\frac{\pi}{4}\leq\angle(X_t,Y_t)\leq\frac{\pi}{4},\forall t<\tau_k\}$.

Our goal is to show that $\P(S)=1$, where $\P$ is the measure induced on
$Z$-pathspace by (\ref{mirrorRSDE}).  It is clear that the subsets $S_k$
decrease monotonically to $S$, and so it suffices to prove that
$\P(S_k)=1,\forall k$.

We first claim that $\P(S_k)=\P^k(S_k)$. This is true because $S_k$ is
$\F_{\tau_k}$-measurable, and, in view of Lemma \ref{lem:ItoisStrat}
and the equality $\sigma=\sigma_k$ on $D_{\frac{1}{k}}$, it is clear that
$\P(A)=\P^k(A)$ for $A\in\F_{\tau_k}$. 
So we need only show that $\P^k(S_k)=1$, and for this it will suffice to
show that $\P^{k,N}(S_k)=1$. We argue this as we did in Example
\ref{lipexample}. 

Fix $N$ and $W_t\in\Omega$ and let $\Theta^{k,N}_t\equiv
\angle(X^{k,N}_t,Y^{k,N}_t)$. By symmetry, it is enough to show that
$\dot{\Theta}^{k,N}_t\leq 0$ for $\Theta^{k,N}_t\in
      [\frac{\pi}{4},\frac{\pi}{2}-\tan^{-1}(\lambda)]$ for almost every
      $t<\tau_k$. Let $v_t=\dot{W}^N_t$ and $$
      w_t=\left(I-\frac{(Y^{k,N}_t-X^{k,N}_t)(Y^{k,N}_t-X^{k,N}_t)^\top}
{|Y^{k,N}_t-X^{k,N}_t|^2}\right)v_t, 
      $$ Then, in view of Theorem \ref{RepresentationTheorem} and Lemma
      \ref{productlemma}, $\dot{X}^{k,N}_t=\tilde{v}_t\equiv\text{{\rm
          proj}}_{T_{\bar{\O}}(X^N_t)}(v_t)$ and
      $\dot{Y}^{k,N}_t=\tilde{w}_t\equiv \text{{\rm
          proj}}_{T_{\bar{\O}}(Y^N_t)}(w_t)$ (recall that
      $\rho_k(X^{k,N}_t,Y^{k,N}_t)=1$ for $t<\tau_k$).

We compute:
\begin{align*}
\dot{\Theta}^{k,N}_t=& \frac{(Y^{N,k}_t-X^{N,k}_t)\cdot
  R(\tilde{v}_t-\tilde{w}_t)}{|Y^{N,k}_t-X^{N,k}_t|^2}\\ =&
\frac{(Y^{N,k}_t-X^{N,k}_t)\cdot
  R(\tilde{v}_t-v_t)}{|Y^{N,k}_t-X^{N,k}_t|^2}+\frac{(Y^{N,k}_t-X^{N,k}_t)\cdot
  R(v_t-w_t)}{|Y^{N,k}_t-X^{N,k}_t|^2}\\ 
&+\frac{(Y^{N,k}_t-X^{N,k}_t)\cdot R(w_t-\tilde{w}_t)}{|Y^{N,k}_t-X^{N,k}_t|^2}
\end{align*}
The argument in the Proof \ref{ProofLip} again shows that the first and
third terms are non-positive. That the second term is non-positive follows
from the fact that either $v_t=w_t$ or $\arg(v_t-w_t)=\pm\Theta^{k,N}_t$.
\end{proof}

%%%%%%%%%%%%%%%%%%%%%%%%%%%%%%%%%%%%%%%%%%%%%%%%%%%%%%%%%%%%%%%%%%%%%%%%%%%%
\bibliography{2ndResult}{}
\bibliographystyle{plain}

\end{document}